\documentclass[11pt]{amsart}

\usepackage[T1]{fontenc}
\usepackage{lmodern}
\usepackage{microtype}
\usepackage{amsmath,amssymb,amsthm,mathtools}
\usepackage[margin=1in]{geometry}
\usepackage{enumitem}
\usepackage{booktabs}
\usepackage{array}
\usepackage{mathrsfs}
\usepackage[colorlinks=true,linkcolor=blue,citecolor=blue,urlcolor=blue]{hyperref}
\usepackage{aliascnt}

\newtheorem{theorem}{Theorem}[section]

\newaliascnt{proposition}{theorem}
\newtheorem{proposition}[proposition]{Proposition}
\aliascntresetthe{proposition}

\newaliascnt{lemma}{theorem}
\newtheorem{lemma}[lemma]{Lemma}
\aliascntresetthe{lemma}

\newaliascnt{corollary}{theorem}
\newtheorem{corollary}[corollary]{Corollary}
\aliascntresetthe{corollary}

\newaliascnt{question}{theorem}
\newtheorem{question}[question]{Question}
\aliascntresetthe{question}

\newaliascnt{externallemma}{theorem}
\newtheorem{externallemma}[externallemma]{External Lemma}
\aliascntresetthe{externallemma}

\theoremstyle{remark}
\newaliascnt{remark}{theorem}
\newtheorem{remark}[remark]{Remark}
\aliascntresetthe{remark}

\usepackage[nameinlink,capitalise]{cleveref}
\crefname{theorem}{Theorem}{Theorems}
\Crefname{theorem}{Theorem}{Theorems}
\crefname{proposition}{Proposition}{Propositions}
\Crefname{proposition}{Proposition}{Propositions}
\crefname{lemma}{Lemma}{Lemmas}
\Crefname{lemma}{Lemma}{Lemmas}
\crefname{corollary}{Corollary}{Corollaries}
\Crefname{corollary}{Corollary}{Corollaries}
\crefname{question}{Question}{Questions}
\Crefname{question}{Question}{Questions}
\crefname{externallemma}{External Lemma}{External Lemmas}
\Crefname{externallemma}{External Lemma}{External Lemmas}
\crefname{remark}{Remark}{Remarks}
\Crefname{remark}{Remark}{Remarks}

\numberwithin{equation}{section}
\allowdisplaybreaks
\setlist[enumerate]{label=(\roman*),leftmargin=2.2em}

\newcommand{\N}{\mathbb N}
\newcommand{\Z}{\mathbb Z}
\newcommand{\E}{\mathbb E}
\newcommand{\Var}{\operatorname{Var}}
\newcommand{\eps}{\varepsilon}
\newcommand{\floor}[1]{\left\lfloor #1\right\rfloor}
\newcommand{\ceil}[1]{\left\lceil #1\right\rceil}
\newcommand{\set}[1]{\left\{#1\right\}}

\title[Prime-power rarefaction and Erd\H{o}s Problem 400]
{Prime-Power Rarefaction and a Density-One Lower Bound for Erd\H{o}s Problem 400}

\author[E. Li]{Eric Li}
\dedicatory{\normalfont\normalsize Trinity College, University of Cambridge}
\date{June 22, 2026}
\thanks{Email addresses: \href{mailto:el593@cam.ac.uk}{el593@cam.ac.uk}, \href{mailto:contact@ericli.com}{contact@ericli.com}.}


\hypersetup{
  pdftitle={Prime-Power Rarefaction and a Density-One Lower Bound for Erdos Problem 400},
  pdfauthor={Eric Li},
  pdfsubject={Density-one bounds for factorial divisibility in Erdos Problem 400},
  pdfkeywords={factorial divisibility, sum of digits, Kummer carries, S-units, p-adic Subspace Theorem, Erdos problems}
}

\begin{document}
\raggedbottom

\begin{abstract}
For fixed $k\ge2$, let $g_k(n)$ be the greatest excess
$a_1+\cdots+a_k-n$ among positive integers $a_i$ satisfying
$a_1!\cdots a_k!\mid n!$.  We prove that, for every $\eps>0$, all but
$o(x)$ integers $n\le x$ satisfy
\[
 g_k(n)\ge\left(\frac{3(k-1)}{\log12}-\eps\right)\log n.
\]
We also prove, as $n\to\infty$, the pointwise upper bound
\[
 g_k(n)\le (k-1)\log_2n+\log_2\log n+O_k(1).
\]
The central analytic input is uniform phase separation for one or two
frequencies on fixed-prime $S$-unit progressions, deduced directly from
the finite exceptional-subspace alternative of Drmota and Spiegelhofer,
and the resulting uniform digit-sum normal-order theorem.  A mixed
$2$--$3$ representation, quantitative two-block estimates, and a
large-prime Kummer sieve produce the stated coefficient.
\end{abstract}

\maketitle
\enlargethispage{2pt}

\section{Introduction}

In their 1980 monograph, Erd\H{o}s and Graham asked for the typical and
average size of an excess arising in products of factorials
\cite[p.~77]{ErdosGraham1980}.  Write
\(\N=\{1,2,3,\ldots\}\).  For a fixed integer \(k\ge2\), define
\begin{equation}\label{eq:defgk}
 g_k(n)=\max\set{a_1+\cdots+a_k-n:
 a_1!\cdots a_k!\mid n!,\ a_i\in\N}.
\end{equation}
The maximum in \eqref{eq:defgk} exists: every admissible \(a_i\) is at
most \(n\), so there are only finitely many admissible tuples, while
\((1,\ldots,1)\) is always admissible.  Erd\H{o}s and Graham asked in
particular whether there is a constant \(c_k\) governing both the
almost-everywhere order of \(g_k(n)\) and the asymptotic behavior of
\(\sum_{n\le x}g_k(n)\).  This is recorded as Erd\H{o}s Problem~400 \cite{Bloom400}.

Here and throughout, a statement holds for \emph{almost all} positive
integers if its exceptional set has asymptotic density zero.  We
establish an explicit density-one lower coefficient, a complementary
pointwise upper bound, and the resulting bracket for the normalized
summatory function.

\subsection{Main results}

The principal result is the following exact density statement.

\begin{theorem}[Density-one lower bound]\label{thm:main}
Fix \(k\ge2\).  For every \(\eps>0\), as \(x\to\infty\),
\[
 \#\set{1\le n\le x:
 g_k(n)<\left(\frac{3(k-1)}{\log12}-\eps\right)\log n}=o(x).
\]
Equivalently, for every \(c<3(k-1)/\log12\), one has
\(g_k(n)\ge c\log n\) for almost all positive integers \(n\).
\end{theorem}

The elementary obstruction coming from binary carries gives the
following pointwise comparison.

\begin{theorem}[Pointwise upper bound]\label{thm:upper}
For every fixed \(k\ge2\), as \(n\to\infty\),
\[
 g_k(n)\le (k-1)\log_2n+\log_2\log n+O_k(1).
\]
Equivalently,
\[
 g_k(n)\le \frac{k-1}{\log2}\log n
 +\frac1{\log2}\log\log n+O_k(1).
\]
\end{theorem}

Summing the density-one lower estimate and the pointwise upper estimate
yields a nontrivial bracket for the normalized average.

\begin{corollary}[Summatory bracket]\label{cor:summatory}
For every fixed \(k\ge2\),
\[
 \frac{3(k-1)}{\log12}
 \le \liminf_{x\to\infty}\frac1{x\log x}\sum_{n\le x}g_k(n)
 \le \limsup_{x\to\infty}\frac1{x\log x}\sum_{n\le x}g_k(n)
 \le \frac{k-1}{\log2}.
\]
\end{corollary}

Numerically,
\[
 \frac3{\log12}=1.2072888131\ldots,
 \qquad \frac1{\log2}=1.4426950409\ldots,
\]
so the proved lower coefficient is
\[
 \frac{3\log2}{\log12}=0.8368288369\ldots
\]
of the binary leading coefficient.  The interval between the two
leading coefficients motivates the questions in the final section.

\subsection{Analytic input and proof architecture}

The fixed-prime distribution theorem needed in the construction is a
normal-order statement along progressions with a growing \(S\)-unit
multiplier.  A positive integer is an \(S\)-unit if all of its prime
factors belong to the finite set \(S\).

\begin{theorem}[Uniform \(S\)-unit digit normal order]\label{thm:sunitnormal}
Let \(S\) and \(\mathcal P\) be disjoint finite sets of primes.  Fix
\(C>0\), \(0<c_0<C_0\), \(c_I>0\), and \(\eps>0\).  Let \(U\to\infty\),
let \(A\) be an \(S\)-unit with \(1\le A\le U^C\), let \(b\in\Z\)
satisfy \(|b|\le(\log U)^C\), and let
\(I\subset[c_0U,C_0U]\) be an interval of at least \(c_IU\)
consecutive integers such that \(Au+b\ge0\) for \(u\in I\).  Uniformly
in \(A,b,I\), all but \(o(U)\) integers \(u\in I\) satisfy,
simultaneously for every \(p\in\mathcal P\),
\begin{equation}\label{eq:sunitnormalintro}
 \left|s_p(Au+b)-\frac{p-1}{2}\log_p(AU)\right|
 \le \eps(p-1)\log_p(AU).
\end{equation}
\end{theorem}

The proof has three main components.
\begin{enumerate}
\item A coefficient-uniform phase-separation lemma for one or two
interior frequencies is derived from the finite exceptional-subspace
alternative of Drmota and Spiegelhofer; see \cref{lem:phase}.
\item Fourier discrepancy estimates and a variance calculation turn
that separation into the uniform fixed-prime normal-order theorem
\cref{thm:sunitnormal}.
\item A mixed binary--ternary representation, full-box affine digit
estimates, and a large-prime Kummer sieve combine the fixed-prime
information into the factorial construction proving \cref{thm:main}.
\end{enumerate}
The logical dependencies may be summarized as
\begin{center}
\small
\begin{tabular}{c}
Subspace alternative $\Longrightarrow$ interior phase separation \\
$\Longrightarrow$ one/two digit patterns $\Longrightarrow$ $S$-unit normal order, \\
interval concentration $+$ quotient pullback $\Longrightarrow$ two-block estimate, \\
representation count $+$ carry propagation \\
$\Longrightarrow$ all-prime factorial divisibility.
\end{tabular}
\end{center}

The coefficient can be read from a two-resource optimization.  Put
\(h=k-1\).  If \(\lambda_2\) and \(\lambda_3\) are the total
logarithmic resources assigned to powers of \(2\) and \(3\), then the
leading binary and ternary digit budgets are
\[
 \frac{h+\lambda_2}{2\log2}\log n,
 \qquad
 \frac{h+\lambda_3}{\log3}\log n.
\]
The exact-representation argument requires
\(\lambda_2+\lambda_3<h\).  At the formal endpoint, balancing both
budgets at \(c\log n\) gives
\[
 \lambda_2=2c\log2-h,
 \qquad
 \lambda_3=c\log3-h,
\]
and exhausting the representation resource yields
\(c\log12=3h\).  The proof works with strict inequalities throughout
and approaches this endpoint only after all error parameters have been
fixed.

\subsection{Relation to previous work}

Pomerance's Theorem~1 \cite{Pomerance2026} proves that, for every fixed
\(\vartheta<1/\log4\), the rising product
\((m+1)\cdots(m+\ell)\) divides \(\binom{2m}{m}\) for all
\(\ell\le\vartheta\log m\) on a density-one set of \(m\).  Appendix
Corollary~2 of Sothanaphan \cite{Sothanaphan2026} gives the same leading
range effectively.  These are centered constructions inside a middle
binomial coefficient.  The construction here is noncentral: it
averages over many exact mixed-radix representations and allocates
binary and ternary digit resources simultaneously.

Related central-binomial work includes the prime-factor analysis of
Erd\H{o}s--Graham--Ruzsa--Straus
\cite{ErdosGrahamRuzsaStraus1975}, Pomerance's earlier divisor theorem
\cite{Pomerance2015}, Sanna's density estimates \cite{Sanna2018}, and
the positive-density divisibility results of Ford--Konyagin
\cite{FordKonyagin2021}.  Croot--Mousavi--Schmidt and Bloom--Croot study
simultaneous low-carry or small-digit restrictions in several bases
\cite{CrootMousaviSchmidt2024,BloomCroot2025}.  On the digital side,
Spiegelhofer's binary--ternary collision theorem and the work of
Spiegelhofer--Stoll on arithmetic progressions provide nearby but
distinct distribution results
\cite{Spiegelhofer2023,SpiegelhoferStoll2020}.

Drmota and Spiegelhofer prove substantially stronger joint
binary--ternary limit theorems, uniformly under the corresponding
prime-power rarefactions \cite{DrmotaSpiegelhofer2025}.  We use their
Lemma~3.3 as the Subspace-Theorem input and derive from it the
coefficient-uniform one- and two-frequency estimate required here.
Because Lemma~3.4 fixes the coefficients before introducing its
threshold, the derivation below supplies the uniformity needed for the
polylogarithmically varying Fourier coefficients.

While this manuscript was in preparation, SamKorsky independently
announced on the Erdős Problems forum \cite{Bloom400} (posts of
16--20 June 2026) the same density-one lower bound with leading
coefficient \(\frac{3(k-1)}{\log 12}\), obtained by optimizing a mixed
binary--ternary resource allocation that reaches the identical natural
endpoint \(c\log 12 = 3(k-1)\) of the method.  The present preprint was
posted on 22 June 2026.  We thank SamKorsky for the parallel progress
and for the valuable discussion thread on the forum.

\subsection{Conventions and organization}

Throughout, \(\log\) denotes the natural logarithm and \(\log_p\) the
logarithm to base \(p\).  For a nonnegative integer \(m\), \(s_p(m)\)
denotes the sum of the base-\(p\) digits of \(m\), and \(\nu_p(m)\)
denotes its \(p\)-adic valuation when \(m\ne0\).  Implied constants may
depend on fixed parameters indicated in their subscripts but not on the
asymptotic variable.  A detailed notation table and the uniformity
conventions for every principal estimate appear in \cref{sec:uniformity}.

The next section fixes notation and uniformity.  The exact valuation
identities and \cref{thm:upper} are proved in Section~3.  Sections~4
and~5 develop the elementary and fixed-prime digit estimates.  Sections
6 and~7 establish the representation and resource-allocation tools,
and Section~8 proves \cref{thm:main}.  The final section records the
summatory consequence, comparisons, and further directions.  Two
appendices give the parameter ledger and an application-by-application
hypothesis audit.

\section{Notation and uniformity conventions}\label{sec:uniformity}

The following symbols are kept distinct throughout.
\begin{center}
\small
\begin{tabular}{@{}p{0.18\textwidth}p{0.72\textwidth}@{}}
\toprule
Symbol & Meaning \\
\midrule
$n$ & factorial target in $g_k(n)$ \\
$N$ & a generic sum $a_1+\cdots+a_k$, or a sample-size parameter only when explicitly stated \\
$X$ & dyadic scale, with $n\asymp X$ \\
$\lambda_2,\lambda_3$ & total binary and ternary logarithmic resources \\
$M_i$ & a pure-power multiplier in the factorial construction \\
$U_i=X/M_i$ & parameter scale associated with $M_i$ \\
$A$ & a generic $S$-unit multiplier in \cref{thm:sunitnormal} \\
$b$ & an additive shift in an affine progression \\
$L$ & a base-$p$ digit length \\
$B$ & exponent in the polylogarithmic prime cutoff $(\log X)^B$ \\
$\mathfrak q(t)$ & an integer quotient map; never a prime \\
\bottomrule
\end{tabular}
\end{center}

The following ledger records the variables over which the principal
estimates are uniform; all data not displayed in its middle column are
fixed before the asymptotic variable tends to infinity.
\begin{center}
\small
\begin{tabular}{@{}p{0.28\textwidth}p{0.42\textwidth}p{0.20\textwidth}@{}}
\toprule
Estimate & Uniform variables & Limit \\
\midrule
\cref{lem:intervalconcentration,lem:targetupper} &
$p\le(\log X)^A$, interval or target & $X\to\infty$ \\
\cref{lem:twoblock} &
$p,M,U,I,d$ satisfying its displayed scale conditions & $X\to\infty$ \\
\cref{lem:phase,prop:patterns} &
$S$-unit $A$, central positions, polylogarithmic frequencies, shift $b$, interval $I$ & $U\to\infty$ \\
\cref{thm:sunitnormal} &
$A,b,I$ and the finite target-prime set $\mathcal P$ & $U\to\infty$ \\
\cref{lem:representations} &
target $m$ throughout one fixed target subinterval & $X\to\infty$ \\
\cref{lem:largeprime} &
all parameter boxes satisfying the fixed lower scale $U_i\ge X^\delta$ & $X\to\infty$ \\
\bottomrule
\end{tabular}
\end{center}
The phrase ``uniformly all but $o(U)$'' has the following fixed meaning:
there is a function $r(U)\to0$, depending only on the fixed data
displayed in the relevant statement, such that every admissible
exceptional set has cardinality at most $r(U)U$.  Analogous language at
scale $X$ is interpreted in the same way.

In the main proof, $k,c$ and all resource exponents are fixed first;
then the finite prime cutoff $P$ and the polylogarithmic exponent $B$
are fixed; only afterward does $X\to\infty$.  No $o(1)$ hides a limit
in $P$, $B$, a target subinterval, or a resource exponent.  The
complete order of parameter choices is repeated in
Appendix~\ref{app:ledger}.

\section{Factorial divisibility, digit sums, and the pointwise upper bound}

We begin with exact identities.  Legendre's formula states that, for every prime $p$ and every nonnegative integer $m$,
\begin{equation}\label{eq:legendre}
 \nu_p(m!)=\frac{m-s_p(m)}{p-1}.
\end{equation}

\begin{lemma}[Exact digit criterion]\label{lem:digitcriterion}
Suppose that $a_1,\ldots,a_k,n,v$ are nonnegative integers satisfying
\[
 a_1+\cdots+a_k=n+v.
\]
Then
\[
 a_1!\cdots a_k!\mid n!
\]
if and only if, for every prime $p$,
\begin{equation}\label{eq:digitcriterion}
 \sum_{i=1}^k s_p(a_i)-s_p(n)\ge v.
\end{equation}
Equivalently,
\begin{equation}\label{eq:risingmultinomial}
 (n+1)(n+2)\cdots(n+v)
 \ \bigg|\ 
 \binom{n+v}{a_1,\ldots,a_k}.
\end{equation}
\end{lemma}

\begin{proof}
By \eqref{eq:legendre} and $\sum_i a_i=n+v$,
\[
 (p-1)\left(\nu_p(n!)-\sum_{i=1}^k\nu_p(a_i!)\right)
 =\sum_{i=1}^k s_p(a_i)-s_p(n)-v.
\]
Thus the factorial divisibility is equivalent to \eqref{eq:digitcriterion} for every prime $p$.  Moreover,
\[
 \frac{1}{(n+1)\cdots(n+v)}
 \binom{n+v}{a_1,\ldots,a_k}
 =\frac{n!}{a_1!\cdots a_k!},
\]
which proves the second equivalence.
\end{proof}

For completeness, we derive the multinomial form of Kummer's theorem
directly.  Write $a_{i,j}$ and $N_j$ for the base-$p$ digits of $a_i$
and $N=\sum_i a_i$, and let $c_j$ be the carry entering column $j$.
Thus $c_0=0$, all but finitely many $c_j$ vanish, and
\[
 c_j+\sum_{i=1}^k a_{i,j}=N_j+pc_{j+1}.
\]
Summing this identity over all columns gives
\begin{equation}\label{eq:carryidentity}
 \sum_{i=1}^k s_p(a_i)-s_p(N)
 =(p-1)\sum_{j\ge1}c_j.
\end{equation}
On the other hand, Legendre's formula gives
\begin{equation}\label{eq:kummerdigit}
 \nu_p\binom{N}{a_1,\ldots,a_k}
 =\frac{\sum_i s_p(a_i)-s_p(N)}{p-1}.
\end{equation}
Consequently the valuation equals the total carry mass
$\sum_{j\ge1}c_j$.  This is the multinomial Kummer theorem, here
obtained from the digit identity; see also Kummer's original work
\cite{Kummer1852}.  We need a particularly simple binary consequence.

\begin{lemma}[Binary carry capacity]\label{lem:binarycapacity}
Let $a_1+\cdots+a_k=N$ and put $L=\floor{\log_2N}$.  Then
\[
 \nu_2\binom{N}{a_1,\ldots,a_k}\le (k-1)L.
\]
\end{lemma}

\begin{proof}
Write $a_{i,j}\in\set{0,1}$ and $N_j\in\set{0,1}$ for the binary digits in column $j$, and let $c_j$ be the carry entering column $j$.  Then $c_0=0$ and
\[
 c_j+\sum_{i=1}^k a_{i,j}=N_j+2c_{j+1}.
\]
Inductively, $0\le c_j\le k-1$ for every $j$.  Since the top nonzero digit of $N$ is in column $L$ and each $a_i\le N$, the carry leaving that column is $c_{L+1}=0$.  Kummer's theorem gives
\[
 \nu_2\binom{N}{a_1,\ldots,a_k}
 =c_1+\cdots+c_L\le (k-1)L.
\]
\end{proof}

\begin{proof}[Proof of \cref{thm:upper}]
Consider an admissible tuple with excess $v\ge0$, and put $N=n+v$.  Since $a_i!\mid n!$, one has $a_i\le n$ for each $i$, and hence $N\le kn$.  By \cref{lem:digitcriterion},
\[
 v\le \sum_i s_2(a_i)-s_2(n).
\]
Insert $s_2(N)$ and apply \cref{lem:binarycapacity}:
\[
 v\le (k-1)\floor{\log_2N}+s_2(N)-s_2(n).
\]
The binary digit sum is subadditive, so $s_2(n+v)\le s_2(n)+s_2(v)$.  Therefore
\begin{equation}\label{eq:upperintermediate}
 v\le (k-1)\log_2(kn)+s_2(v).
\end{equation}
For $v\ge1$, the elementary estimate $s_2(v)\le 1+\log_2v$ first implies from \eqref{eq:upperintermediate} that $v=O_k(\log n)$.  Substituting this information back into \eqref{eq:upperintermediate} gives
\[
 v\le (k-1)\log_2n+\log_2\log n+O_k(1).
\]
Taking the maximum over admissible tuples completes the proof.
\end{proof}

\section{Quantitative elementary digit estimates}

This section makes the elementary estimates sufficiently quantitative
to be summed over every prime below a polylogarithmic cutoff.  The
basic input is Hoeffding's inequality \cite{Hoeffding1963}.

\begin{lemma}[Digit-word concentration]\label{lem:wordconcentration}
For every $\eps>0$ there is $c_\eps>0$ such that, for every prime $p$
and every integer $L\ge1$, the number of $0\le m<p^L$ satisfying
\[
 \left|s_p(m)-\frac{p-1}{2}L\right|>\eps(p-1)L
\]
is at most $2p^L\exp(-c_\eps L)$.
\end{lemma}

\begin{proof}
For uniform $m\in[0,p^L)$, the $L$ base-$p$ digits are independent and
uniform on $\{0,\ldots,p-1\}$.  Divide each digit by $p-1$ and apply
Hoeffding's inequality.
\end{proof}

\begin{lemma}[Quantitative concentration on intervals]
\label{lem:intervalconcentration}
Fix $A,\delta,\eps,D>0$.  There is $c>0$ with the following property.
Let $p\le(\log X)^A$, and let $J\subset\Z_{\ge0}$ be an interval of
$H\ge X^\delta$ consecutive integers.  If $E_p(J)$ is the set of
$m\in J$ for which
\begin{equation}\label{eq:intervallower}
 s_p(m)<\left(\frac12-\eps\right)(p-1)\log_pH,
\end{equation}
then
\begin{equation}\label{eq:intervalquant}
 \frac{|E_p(J)|}{H}
 \ll (\log X)^{-D}
 +\exp\!\left(-c\frac{\log X}{\log\log X}\right).
\end{equation}
The implied constant is uniform in $p$ and $J$.  Consequently, if
$D>A+2$, then
\[
 \sum_{p\le(\log X)^A}\frac{|E_p(J)|}{H}=o(1)
\]
uniformly over all such intervals.
\end{lemma}

\begin{proof}
Choose
\[
 R>D+A+4
\]
and let $P=p^L$ be the largest power of $p$ not exceeding
$H/(\log X)^R$.  Thus
\begin{equation}\label{eq:Pinterval}
 \frac{H}{p(\log X)^R}<P\le\frac{H}{(\log X)^R}.
\end{equation}
Except for two boundary pieces of total length less than $2P$, the
interval $J$ is a union of complete blocks of length $P$.  On every
complete block the lowest $L$ digits run through all base-$p$ words of
length $L$.  By \cref{lem:wordconcentration}, all but a proportion
$2e^{-c_\eps L}$ of each complete block have lower-digit sum at least
$(1/2-\eps/2)(p-1)L$.

From \eqref{eq:Pinterval},
\[
 L=\log_pH+O\!\left(\frac{R\log\log X}{\log p}+1\right).
\]
Because $H\ge X^\delta$ and $p\le(\log X)^A$,
\[
 L\ge \frac{\delta}{2A}\frac{\log X}{\log\log X}
\]
for all sufficiently large $X$ (with the evident easier interpretation
when $A<1$).  The difference between $L$ and $\log_pH$ is therefore
$o(\log_pH)$, uniformly in $p$.  Hence the preceding lower-digit bound
implies \eqref{eq:intervallower} for large $X$.

The boundary proportion is $O((\log X)^{-R})$, and the concentration
tail is
\[
 \ll\exp\!\left(-c\frac{\log X}{\log\log X}\right).
\]
This proves \eqref{eq:intervalquant}, after weakening the first term to
$(\log X)^{-D}$.  There are at most $(\log X)^A$ relevant primes, so
summation gives the final assertion.
\end{proof}

\begin{lemma}[Simultaneous upper bound for targets]
\label{lem:targetupper}
Fix $A,\eps>0$ and $0<c_0<C_0$.  There is $c>0$ such that all but
\[
 \ll X\exp\!\left(-c\frac{\log X}{\log\log X}\right)=o(X)
\]
integers $n\in[c_0X,C_0X]$ satisfy, simultaneously for every prime
$p\le(\log X)^A$,
\begin{equation}\label{eq:targetupper}
 s_p(n)\le\left(\frac12+\eps\right)(p-1)\log_pX.
\end{equation}
\end{lemma}

\begin{proof}
For a fixed $p$, pad each $n<C_0X$ with leading zero digits to length
$L=\ceil{\log_p(C_0X)}$.  Lemma~\ref{lem:wordconcentration} gives an
exceptional proportion $\ll p e^{-c_\eps L}$.  Uniformly for
$p\le(\log X)^A$, one has
$L\gg_A\log X/\log\log X$.  The bounded change from $L$ to
$\log_pX$ is absorbed by the fixed $\eps$ margin.  The factors coming from padding, the number of primes, and the largest
base are fixed powers of $\log X$; reducing $c$ absorbs all of them into
the displayed exponential.  Summing over the indicated primes proves
the result.
\end{proof}

The next lemma is the growing-base estimate used later.  Its two digit
blocks are separated explicitly, and its error is strong enough for a
union over polylogarithmically many primes.

\begin{lemma}[Pullback through a monotone quotient map]\label{lem:pullback}
Let $I$ be an interval of $T$ consecutive integers, let $0<M<Q$, and
put
\[
 \mathfrak q(t)=\floor{\frac{Mt+d}{Q}}.
\]
Its image $J$ is an interval of consecutive integers.  If $H=|J|$,
then every non-endpoint value has $Q/M+O(1)$ preimages, and for every
subset $E\subset J$,
\begin{equation}\label{eq:pullback}
 \frac{|\{t\in I:\mathfrak q(t)\in E\}|}{T}
 \le \frac{|E|}{H}+O\!\left(\frac MQ+\frac1H\right).
\end{equation}
The implied constant is absolute.
\end{lemma}

\begin{proof}
For a fixed integer $y$, the condition $\mathfrak q(t)=y$ cuts out an
interval of real length $Q/M$, so its number of integer points is
$Q/M+O(1)$ unless it is truncated by an endpoint of $I$.  At most two
image values are truncated.  Consequently
\[
 |\mathfrak q^{-1}(E)|=|E|\frac QM+O\!\left(|E|+\frac QM\right)
\]
and
\[
 T=H\frac QM+O\!\left(H+\frac QM\right).
\]
Divide the two relations and use $|E|\le H$.
\end{proof}

\begin{lemma}[Quantitative two-block affine estimate]
\label{lem:twoblock}
Fix $A_0,A_1,D,\eps,\delta_0>0$ and positive constants
$c_I,C_I,c_X,C_X$.  Let $a,u\in(0,1)$ satisfy
\[
 a+u=1,
 \qquad \min\{a,u,|a-u|\}\ge\delta_0.
\]
Suppose that $M$ is a positive integer and $U$ is a positive real number satisfying
\begin{align*}
 X^a(\log X)^{-A_1}&\le M\le X^a(\log X)^{A_1},\\
 X^u(\log X)^{-A_1}&\le U\le X^u(\log X)^{A_1}.
\end{align*}
and $MU\asymp X$, with constants independent of $X$.  Let $I$ be an
interval of consecutive integers with
$c_IU\le|I|\le C_IU$, and let $d\in\Z$ satisfy
$|d|\le(\log X)^{A_1}$ and
\[
 c_XX\le Mt+d\le C_XX\qquad(t\in I).
\]
For every prime $p\le(\log X)^{A_0}$ with $p\nmid M$, let
\[
 E_p=\left\{t\in I:
 \frac{s_p(Mt+d)}{p-1}<
 \left(\min\left\{u,\frac12\right\}-\eps\right)\log_pX
 \right\}.
\]
Then there is $c>0$ such that, uniformly in all the displayed data,
\begin{equation}\label{eq:twoblockquant}
 \frac{|E_p|}{U}
 \ll (\log X)^{-D}
 +\exp\!\left(-c\frac{\log X}{\log\log X}\right).
\end{equation}
In particular, when $D>A_0+2$, the union of the $E_p$ over
$p\le(\log X)^{A_0}$ has size $o(U)$.  The statement includes the
shift $d=-1$ used in the factorial construction.
\end{lemma}

\begin{proof}
All constants below may depend on the fixed data in the statement.
Choose a constant
\[
 R>D+A_0+A_1+10
\]
large enough that every loss of $O(R\log\log X)$ base-$p$ digits is
at most $(\eps/20)\log_pX$ for large $X$.  Whenever a power of $p$ is
chosen at a prescribed scale, choosing the largest power below that
scale or the smallest power above it changes the scale by a factor at
most $p\le(\log X)^{A_0}$.  Thus every logarithmic block length changes
by at most one base-$p$ digit, which is
$O_{A_0}(\log\log X/\log p)$ in the scale notation used below.  The
same bound absorbs the polylogarithmic uncertainty in $M$ and $U$.
The additive shift $d$, including $d=-1$, merely changes the initial
residue and the two endpoints of a quotient interval by $O(1)$.

For reference, the scales used in the two regimes are
\begin{center}
\small
\begin{tabular}{@{}>{\raggedright\arraybackslash}p{0.22\textwidth}
>{\centering\arraybackslash}p{0.31\textwidth}
>{\centering\arraybackslash}p{0.31\textwidth}@{}}
\toprule
Quantity & $a<u$ & $a>u$ \\
\midrule
low modulus & $Q=X^{1/2}(\log X)^{-R+O(A_0)}$
 & $Q_-=U(\log X)^{-R+O(A_0)}$ \\
boundary loss & $Q/U=X^{-\delta_0/2+o(1)}$
 & $Q_-/U=(\log X)^{-R+O(A_0)}$ \\
high modulus & $Q$ & $Q_+=M(\log X)^{R+O(A_0)}$ \\
multiplicity loss & $M/Q=X^{-\delta_0/2+o(1)}$
 & $M/Q_+=(\log X)^{-R+O(A_0)}$ \\
quotient length & $H_Q=X^{1/2}(\log X)^{R+O(A_0)}$
 & $H_+=U(\log X)^{-R+O(A_0)}$ \\
counted length & $\log_pX+O(\log\log X/\log p)$
 & $2u\log_pX+O(\log\log X/\log p)$ \\
\bottomrule
\end{tabular}
\end{center}
In particular, throughout the proof the quotient interval has length
at least $X^{\delta'}$, where one may take
$\delta'=\min\{1/3,\delta_0/2\}>0$.

\smallskip
\noindent\emph{First regime: $a<u$.}
Here $a\le1/2-\delta_0/2$ and $u\ge1/2+\delta_0/2$.  Let $Q$ be the
largest power of $p$ satisfying
\[
 Q\le X^{1/2}(\log X)^{-R}.
\]
Then
\[
 \frac QU+\frac MQ\ll X^{-\delta_0/3}
\]
for large $X$, uniformly in $p$.  Because $M$ is invertible modulo
$Q$, as $t$ runs through any complete block of $Q$ consecutive
integers the residues $Mt+d\pmod Q$ run once through all residues.
The incomplete initial and final blocks contribute $O(Q/U)$ of the
parameters.  Lemma~\ref{lem:wordconcentration}, applied to the
$\log_pQ$ low digits, therefore shows that outside a proportion
\begin{equation}\label{eq:lowerrorfirst}
 O(Q/U)+
 \exp\!\left(-c\frac{\log X}{\log p}\right)
\end{equation}
one has
\begin{equation}\label{eq:lowfirst}
 s_p((Mt+d)\bmod Q)
 \ge\left(\frac12-\frac\eps4\right)(p-1)\log_pQ.
\end{equation}

Write
\[
 Mt+d=Q\,\mathfrak q(t)+r(t),
 \qquad 0\le r(t)<Q.
\]
Since $M<Q$, the nondecreasing map $t\mapsto\mathfrak q(t)$ has a
consecutive image interval $J_Q$ of length
\[
 H_Q\asymp\frac{MU}{Q}
 \asymp X^{1/2}(\log X)^{R+O(A_0)},
\]
and each interior value has $Q/M+O(1)$ preimages.  Thus the relative
variation in the multiplicity is $O(M/Q)$.  Apply
\cref{lem:intervalconcentration} to $J_Q$, with an arbitrarily large
logarithmic exponent, and pull the result back to $I$ using
\cref{lem:pullback}.  Outside a
proportion
\begin{equation}\label{eq:higherrorfirst}
 O(M/Q)+O(1/H_Q)+O((\log X)^{-D-2A_0-4})
 +\exp\!\left(-c\frac{\log X}{\log\log X}\right)
\end{equation}
one has
\begin{equation}\label{eq:highfirst}
 s_p(\mathfrak q(t))
 \ge\left(\frac12-\frac\eps4\right)(p-1)\log_pH_Q.
\end{equation}
Here $H_Q\ge X^{1/3}$ for all sufficiently large $X$, so the
hypothesis of \cref{lem:intervalconcentration} is uniform.  Because
$Q$ is a power of $p$, the identity
$Mt+d=Q\mathfrak q(t)+r(t)$ concatenates the two digit blocks.  Since
$QH_Q\asymp MU\asymp X$, the sum of
\eqref{eq:lowfirst} and \eqref{eq:highfirst}, with the
$O(\log\log X)$ scale losses absorbed, is at least
\[
 \left(\frac12-\eps\right)(p-1)\log_pX.
\]
This is the required bound because $\min\{u,1/2\}=1/2$.

\smallskip
\noindent\emph{Second regime: $a>u$.}
Now $a\ge1/2+\delta_0/2$ and $u\le1/2-\delta_0/2$.  Let $Q_-$ be the
largest power of $p$ below $U(\log X)^{-R}$, and let $Q_+$ be the
smallest power of $p$ above $M(\log X)^R$.  Then
\[
 \frac{Q_-}{U}\ll(\log X)^{-R},
 \qquad
 \frac{M}{Q_+}\ll(\log X)^{-R},
\]
and $Q_-<Q_+$ for large $X$.  Complete residue blocks modulo $Q_-$
give, outside a proportion
\begin{equation}\label{eq:lowerrorsecond}
 O(Q_-/U)+
 \exp\!\left(-c\frac{\log X}{\log p}\right),
\end{equation}
the estimate
\begin{equation}\label{eq:lowsecond}
 s_p((Mt+d)\bmod Q_-)
 \ge\left(\frac12-\frac\eps4\right)(p-1)\log_pQ_-.
\end{equation}

Define
\[
 \mathfrak q_+(t)=\floor{\frac{Mt+d}{Q_+}}.
\]
Its image is a consecutive interval of length
\[
 H_+\asymp\frac{MU}{Q_+}
 \asymp U(\log X)^{-R+O(A_0)},
\]
and $H_+\ge X^{\delta_0/2}$ for all sufficiently large $X$.  Every
interior value has $Q_+/M+O(1)$ preimages, so the relative multiplicity
error is $O(M/Q_+)$.  Quantitative interval concentration, followed by
\cref{lem:pullback}, has total exceptional proportion
\begin{equation}\label{eq:higherrorsecond}
 O(M/Q_+)+O(1/H_+)+O((\log X)^{-D-2A_0-4})
 +\exp\!\left(-c\frac{\log X}{\log\log X}\right).
\end{equation}
Outside that set one has
\begin{equation}\label{eq:highsecond}
 s_p(\mathfrak q_+(t))
 \ge\left(\frac12-\frac\eps4\right)(p-1)\log_pH_+.
\end{equation}
The digits counted in \eqref{eq:lowsecond} lie strictly below the
$Q_-$ place.  The digits of $\mathfrak q_+(t)$ counted in
\eqref{eq:highsecond} occur at and above the $Q_+$ place in $Mt+d$.
Moreover,
\[
 \frac{Q_+}{Q_-}
 \gg \frac{M}{U}(\log X)^{2R-O(A_0)}
 =X^{a-u+o(1)}\to\infty,
\]
so the two blocks are genuinely disjoint with a growing intervening
gap.  Since
\[
 \log_pQ_-+\log_pH_+
 =2u\log_pX+O\!\left(\frac{\log\log X}{\log p}\right),
\]
their combined contribution is at least
\[
 (u-\eps)(p-1)\log_pX.
\]
This is the desired estimate in the second regime.

In both regimes the complete-block boundary errors, the $1/H$ endpoint
terms from \cref{lem:pullback}, nearest-power rounding losses, quotient
multiplicity errors, and concentration tails have been displayed
separately.  The power-saving errors are smaller than any fixed
negative power of $\log X$; the remaining logarithmic errors are made
$O((\log X)^{-D})$ by the choice of $R$.  Each replacement of a
requested scale by a neighboring power of $p$ costs at most one digit,
and all such costs together are
$O_{A_0,A_1,R}(\log\log X/\log p)=o(\log_pX)$ uniformly.
This proves \eqref{eq:twoblockquant}.  Summing over at most
$(\log X)^{A_0}$ primes proves the simultaneous assertion.
\end{proof}

\section{Fixed-prime digit mass on S-unit progressions}\label{sec:sunit}

This section proves \cref{thm:sunitnormal}.  Only one- and two-digit
correlations are needed.  We therefore work directly with the finite
exceptional-subspace alternative supplied by the $p$-adic Subspace
Theorem, rather than with a stronger coefficient-uniform reading of a
later separation lemma.

For a finite place $\ell$, normalize $|\ell|_\ell=\ell^{-1}$, and use
the usual absolute value at $\infty$.

\begin{externallemma}[Subspace-Theorem alternative]\label{lem:subspacealt}
Fix a finite set of places $T$ containing $\infty$, an integer
$d\ge1$, and $\rho>0$.  There exists a finite family $\mathscr W$ of
proper rational subspaces of $\mathbb Q^{d+1}$ such that every
primitive vector $\mathbf x=(x_1,\ldots,x_{d+1})\in\mathbb Z^{d+1}$
outside $\bigcup_{W\in\mathscr W}W$ satisfies
\begin{equation}\label{eq:subspacealt}
 \prod_{v\in T}
 \left|(x_1+\cdots+x_{d+1})x_1\cdots x_{d+1}\right|_v
 >\|\mathbf x\|_\infty^{1-\rho}.
\end{equation}
The family depends only on $T,d,\rho$.
\end{externallemma}

\begin{proof}[Source]
This is the specialization of Drmota--Spiegelhofer
\cite[Lemma~3.3]{DrmotaSpiegelhofer2025}, whose immediate source is
the $p$-adic Subspace Theorem in the form cited there as
Evertse \cite[Chapter~IV, Theorem~1.8]{Evertse1993}.  At every place
we take the $d+1$ coordinate forms and their sum.  These $d+2$ forms
are in general position: any $d+1$ of them are linearly independent,
because either they are all coordinate forms or the sum replaces one
coordinate form.  Here $r=d+2$ and $n=d+1$, so the exponent in the
quoted inequality is $r-n-\rho=1-\rho$.  The exceptional subspaces
depend only on the fixed forms, places, dimension, and $\rho$, not on
the integer coordinates.
\end{proof}

\begin{lemma}[$S$-unit sparse-form alternative]\label{lem:sparsealternative}
Fix a finite set of primes $S$, a prime $p\notin S$, an integer
$d\ge1$, and $\rho>0$.  There is a finite family $\mathscr L$ of
proper rational subspaces of $\mathbb Q^{d+1}$ with the following
property.  Let $A$ be an $S$-unit, let
\[
 m\ge m_1>\cdots>m_d=0,
\]
and let $h_1,\ldots,h_d,H$ be nonzero integers with $p\nmid h_d$.
Put
\[
 x_i=Ap^{m_i}h_i\quad(1\le i\le d),
 \qquad x_{d+1}=-p^mH,
 \qquad F=x_1+\cdots+x_{d+1}.
\]
Then either $\mathbf x\in\bigcup_{W\in\mathscr L}W$, or
\begin{equation}\label{eq:sparsealternative}
 |F|\,|h_1\cdots h_dH|
 \ge \|\mathbf x\|_\infty^{1-\rho}.
\end{equation}
The family and the bound are uniform in $A$, the exponents, and all
integer coefficients.
\end{lemma}

\begin{proof}
Apply \cref{lem:subspacealt} with
$T=\{\infty,p\}\cup S$.  We spell out the primitive normalization,
because it is part of the required uniformity.  Let
$g=\gcd(x_1,\ldots,x_{d+1})$.  Since $p\nmid Ah_d$, one has $p\nmid g$.
Set
\[
 g_A=\gcd(g,A),\qquad g_0=g/g_A.
\]
Prime by prime, $g_0$ divides every $h_i$, while the full divisor
$g=g_Ag_0$ divides $H$.  Thus
\[
 A'=A/g_A,\qquad h_i'=h_i/g_0,\qquad H'=H/g
\]
are integers, $A'$ is an $S$-unit, and the primitive vector
$\mathbf y=\mathbf x/g$ has the form
\[
 y_i=A'p^{m_i}h_i',\qquad y_{d+1}=-p^mH'.
\]
Because $p\nmid g_0$ and $p\nmid h_d$, primitive normalization
preserves the crucial condition $p\nmid h_d'$.
If $\mathbf y$ belongs to an exceptional subspace then, by homogeneity,
so does $\mathbf x$.  Otherwise \eqref{eq:subspacealt} applies.

For each coordinate, the factor $A'p^{m_i}$ is a $T$-unit, and hence
the product formula gives
\[
 \prod_{v\in T}|y_i|_v\le |h_i'|,
 \qquad
 \prod_{v\in T}|y_{d+1}|_v\le |H'|.
\]
At every finite place, $|y_1+\cdots+y_{d+1}|_v\le1$.  Therefore
\[
 |F/g|\,|h_1'\cdots h_d'H'|
 \ge \|\mathbf y\|_\infty^{1-\rho}.
\]
Since $h_i=g_0h_i'$, $H=gH'$, and
$\|\mathbf y\|_\infty=\|\mathbf x\|_\infty/g$, multiplication by
$g^2g_0^d$ yields
\[
 |F|\,|h_1\cdots h_dH|
 \ge g^{1+\rho}g_0^d\|\mathbf x\|_\infty^{1-\rho}
 \ge \|\mathbf x\|_\infty^{1-\rho}.
\]
\end{proof}

For $z\in\mathbb R$, let $\|z\|$ denote the distance to the nearest
integer and put $e(z)=e^{2\pi iz}$.

\begin{lemma}[Uniform interior phase separation]\label{lem:phase}
Fix a finite set of primes $S$, a prime $p\notin S$, constants
$C,R>0$ and $0<\eta<1/4$, and $d\in\{1,2\}$.  There exist
$\kappa>0$ and $U_0$ such that the following holds for $U\ge U_0$.
Let $A$ be an $S$-unit with $1\le A\le U^C$, put
$L=\floor{\log_p(AU)}$, choose positions
\[
 \eta L\le j_1<\cdots<j_d\le(1-\eta)L,
\]
and choose nonzero integers $h_i$ with
$p\nmid h_i$ and $|h_i|\le(\log U)^R$.  Then
\[
 \alpha=A\sum_{i=1}^d h_ip^{-j_i-1}
\]
satisfies
\begin{equation}\label{eq:phase}
 \|\alpha\|\ge U^{-1+\kappa}.
\end{equation}
All constants are uniform in $A$, the positions, and the
coefficients.
\end{lemma}

\begin{proof}
Put
\[
 m=j_d+1,\qquad m_i=j_d-j_i,
 \qquad P=\sum_{i=1}^d h_ip^{m_i}.
\]
Then $m_1>\cdots>m_d=0$ and $p\nmid P$, so $P\ne0$.  Let $H$ be a
nearest integer to $AP/p^m$ and set $F=AP-p^mH$.  The integer $F$ is
nonzero, because $p\nmid AP$ and $m\ge1$.

If $H=0$, then
\[
 \|\alpha\|=\frac{A|P|}{p^m}\ge\frac A{p^m}.
\]
Since $m\le(1-\eta)L+1$ and $p^L\le AU$,
\[
 \frac A{p^m}\ge p^{-1}A^\eta U^{-1+\eta}
 \ge p^{-1}U^{-1+\eta}.
\]
Thus \eqref{eq:phase} holds after reducing $\kappa$.

Assume $H\ne0$.  Apply \cref{lem:sparsealternative} with a parameter
$\rho>0$ to
\[
 (Ap^{m_1}h_1,\ldots,Ah_d,-p^mH).
\]
Let $M$ be its sup norm.  Outside the finite exceptional subspaces,
\eqref{eq:sparsealternative} gives
\[
 \frac{|F|}{p^m}
 \ge \frac{M^{1-\rho}}{p^m|H|\,|h_1\cdots h_d|}
 \ge \frac{M^{-\rho}}{|h_1\cdots h_d|},
\]
because $p^m|H|\le M$.  There is a constant $C_1$, depending only on
the fixed data, for which $M\le U^{C_1}$: indeed,
$p^m\ll(AU)^{1-\eta}$, the sparse coefficients are polylogarithmic,
and
\[
 p^m|H|\le A|P|+\tfrac12p^m.
\]
Choose $\rho>0$ so small that $C_1\rho<1/4$.  The coefficient product
is at most $U^{1/4}$ for large $U$, so the last lower bound is at least
$U^{-1/2}$.  Hence \eqref{eq:phase} holds in the nonexceptional case.

It remains to analyze the finitely many exceptional subspaces.  This
is where coefficient-uniformity is proved directly.  For each proper
subspace $W$ in the fixed exceptional family, choose once and for all
a rational hyperplane containing $W$ and clear denominators in its
equation.  Primitive normalization preserves the shape
\[
 (A'p^{m_1}h_1',\ldots,A'h_d',-p^mH').
\]
Here $A'$ is an $S$-unit.  In the present phase lemma every original
$h_i$ is prime to $p$, while the normalizing divisor is prime to $p$;
therefore
\begin{equation}\label{eq:primitivepfree}
 p\nmid h_i'\qquad(1\le i\le d).
\end{equation}

We first complete the case $d=1$.  A containing hyperplane has an
equation $c_1X_1+c_2X_2=0$ with fixed integer coefficients not both
zero.  If $c_1=0$, then $c_2X_2=0$ at a point with $X_2\ne0$, forcing
$c_2=0$, a contradiction; the case $c_2=0$ is identical.  Thus both
coefficients are nonzero, and the relation is
\[
 c_1A'h_1'=c_2p^mH'.
\]
Using \eqref{eq:primitivepfree} and $p\nmid A'$, comparison of
$p$-adic valuations gives
\[
 m=\nu_p(c_1)-\nu_p(c_2)-\nu_p(H')\le\nu_p(c_1).
\]
This is impossible for all $U\ge U_0(W)$, because
$m\ge\eta L$ and $L\ge\log_pU-1$.  Hence no exceptional hyperplane
contains an admissible $d=1$ point once $U$ is large enough.

Now let $d=2$ and write $a=m_1=j_2-j_1\ge1$.  A fixed containing
hyperplane has equation
\begin{equation}\label{eq:exceptionalplane}
 c_1A'p^ah_1'+c_2A'h_2'-c_3p^mH'=0.
\end{equation}
There are exactly two coefficient cases to consider.  When $c_2\ne0$,
either or both of $c_1,c_3$ may vanish; the reduction below does not
use them.  When $c_2=0$, nontriviality of the hyperplane and the fact
that all three coordinates are nonzero force $c_1c_3\ne0$.

If $c_2\ne0$, reduction of \eqref{eq:exceptionalplane} modulo $p^a$
is legitimate because $m>a$ and gives
\[
 a\le\nu_p(c_2),
\]
using \eqref{eq:primitivepfree}.  Thus $a$ is bounded solely in terms
of this fixed hyperplane.  In the original, unnormalized phase group
the sparse numerator is
\[
 \widetilde h=p^ah_1+h_2.
\]
This integer is nonzero and is prime to $p$, since
$\widetilde h\equiv h_2\not\equiv0\pmod p$.  Moreover,
\[
 |\widetilde h|
 \le (p^{\nu_p(c_2)}+1)(\log U)^R
 \le(\log U)^{R_W}
\]
for one fixed exponent $R_W$.  The original phase is therefore the
one-frequency phase $A\widetilde h/p^m$ at the position
$m-1=j_2$, which remains in the original central window.  The already
proved $d=1$ case, with coefficient exponent $R_W$, yields
\eqref{eq:phase} with constants $\kappa_W>0$ and $U_0(W)$.

If $c_2=0$, use the already noted fact $c_1c_3\ne0$.  Using \eqref{eq:primitivepfree}, comparison of $p$-adic valuations in
\eqref{eq:exceptionalplane} gives
\[
 m-a=\nu_p(c_1)-\nu_p(c_3)-\nu_p(H')
 \le \nu_p(c_1)-\nu_p(c_3).
\]
But $m-a=j_1+1\ge\eta L$, which is impossible for
$U\ge U_0(W)$.

Finally, the nonexceptional branch gives exponent $1/2$, while the
$H=0$ branch gives any exponent smaller than $\eta$.  If $d=1$, take
$\kappa=\min\{\eta/2,1/2\}$ and the maximum of the finitely many
exclusion thresholds.  If $d=2$, define explicitly
\[
 \kappa=
 \min\left\{\frac\eta2,\frac12,
       \min_{W:\,c_2(W)\ne0}\kappa_W\right\}>0,
 \qquad
 U_0=\max_{W\in\mathscr L}U_0(W),
\]
with the inner minimum omitted if no reducing hyperplane occurs.
Enlarging $U_0$ to absorb the finitely many preliminary thresholds
proves the uniform statement.
\end{proof}

\begin{proposition}[Separated one- and two-digit patterns]\label{prop:patterns}
Fix the data $S,p,C,c_0,C_0,c_I$ with $p\notin S$, and fix
$0<\eta<1/4$ and $M_*>0$.  Let $A$ be an $S$-unit with
$1\le A\le U^C$, let $|b|\le(\log U)^C$, and let
$I\subset[c_0U,C_0U]$ be an interval of at least $c_IU$ consecutive
integers on which $Au+b\ge0$.  Put $L=\floor{\log_p(AU)}$.

Uniformly for every central position
$\eta L\le j\le(1-\eta)L$ and digit $a\in\{0,\ldots,p-1\}$,
\begin{equation}\label{eq:onepattern}
 \#\{u\in I:\varepsilon_{p,j}(Au+b)=a\}
 =\frac{|I|}{p}+O\!\left(\frac{U}{(\log U)^{M_*}}\right).
\end{equation}
There is $G>0$ such that, uniformly for central positions
$j_1<j_2$ satisfying $j_2-j_1\ge G\log\log U$ and digits $a_1,a_2$,
\begin{equation}\label{eq:twopattern}
 \#\{u\in I:\varepsilon_{p,j_i}(Au+b)=a_i\ (i=1,2)\}
 =\frac{|I|}{p^2}+O\!\left(\frac{U}{(\log U)^{M_*}}\right).
\end{equation}
\end{proposition}

\begin{proof}
The digit condition $\varepsilon_{p,j}(x)=a$ is equivalent to
\[
 \left\{\frac{x}{p^{j+1}}\right\}\in
 \left[\frac ap,\frac{a+1}{p}\right).
\]
Apply the Erd\H{o}s--Tur\'an--Koksma inequality
\cite[Chapter~2, Theorem~2.5]{KuipersNiederreiter1974} to the one- or
two-dimensional point set
\[
 \left(\left\{\frac{Au+b}{p^{j_1+1}}\right\},\ldots,
 \left\{\frac{Au+b}{p^{j_d+1}}\right\}\right),
 \qquad u\in I,
\]
where $d=1$ or $2$.  Use a frequency cutoff
$H_0=(\log U)^{M_0}$.  The unnormalized discrepancy for an
axis-parallel box is
\begin{equation}\label{eq:etk}
 \ll \frac U{H_0}+
 \sum_{0<\|\mathbf h\|_\infty\le H_0}
 \frac1{r(\mathbf h)}
 \left|\sum_{u\in I}
 e\!\left(uA\sum_{i=1}^dh_ip^{-j_i-1}\right)\right|,
\end{equation}
where $r(\mathbf h)=\prod_i\max(1,|h_i|)$; the shift $b$ contributes
only a unimodular factor.

For a nonzero frequency, remove zero coordinates, write each
remaining $h_i=p^{\ell_i}h_i'$ with $p\nmid h_i'$, and shift the
corresponding position from $j_i$ to $j_i-\ell_i$.  Since
$\ell_i=O(\log\log U)$, all shifted positions remain in a central
window with $\eta$ replaced by $\eta/2$.  In the two-dimensional
case, choose $G$ larger than twice the maximal possible shift
constant.  Then the shifted positions remain distinct.  By
\cref{lem:phase},
\[
 \left\|A\sum_i h_ip^{-j_i-1}\right\|\ge U^{-1+\kappa}.
\]
The geometric-sum bound therefore yields
\[
 \left|\sum_{u\in I}e(u\alpha)\right|
 \le \min\{|I|,(2\|\alpha\|)^{-1}\}
 \ll U^{1-\kappa}.
\]
Also
$\sum_{\|\mathbf h\|_\infty\le H_0}r(\mathbf h)^{-1}
\ll(\log H_0)^d$.  Substitution in \eqref{eq:etk} gives
\[
 O\!\left(\frac U{H_0}+U^{1-\kappa}(\log H_0)^2\right).
\]
Choose $M_0>M_*+2$; the second term is smaller than
$U(\log U)^{-M_*}$ for large $U$.  This proves both estimates.
\end{proof}

\begin{lemma}[Uniform digit means and covariances]
\label{lem:digitcovariance}
Under the hypotheses of \cref{prop:patterns}, fix $M>0$.  Uniformly in
all admissible $A,b,I$, every central position $j$ satisfies
\begin{equation}\label{eq:digitmean}
 \E\,\varepsilon_{p,j}(Au+b)
 =\frac{p-1}{2}+O_{p,M}((\log U)^{-M}).
\end{equation}
If $j_1<j_2$ are central and
$j_2-j_1\ge G\log\log U$, where $G$ is chosen as in
\cref{prop:patterns}, then
\begin{equation}\label{eq:digitcovariance}
 \operatorname{Cov}\!\left(
 \varepsilon_{p,j_1}(Au+b),
 \varepsilon_{p,j_2}(Au+b)
 \right)
 =O_{p,M}((\log U)^{-M}).
\end{equation}
Here expectation is with respect to the uniform probability measure on
$I$.
\end{lemma}

\begin{proof}
Apply \eqref{eq:onepattern} with error exponent $M+2$.  Since
$|I|\ge c_IU$, division by $|I|$ gives, uniformly for every digit $a$,
\[
 \mathbb P(\varepsilon_{p,j}=a)
 =\frac1p+O((\log U)^{-M-2}).
\]
Multiplication by $a$ and summation over the fixed set
$0\le a<p$ proves \eqref{eq:digitmean}.  Likewise,
\eqref{eq:twopattern} gives
\[
 \mathbb P(\varepsilon_{p,j_1}=a_1,
            \varepsilon_{p,j_2}=a_2)
 =\frac1{p^2}+O((\log U)^{-M-2})
\]
uniformly in $a_1,a_2$.  Summing with weights $a_1a_2$ and subtracting
the product of the two means proves \eqref{eq:digitcovariance}.
\end{proof}

\begin{proof}[Proof of \cref{thm:sunitnormal}]
It suffices to prove the assertion for one fixed $p\in\mathcal P$;
a finite union gives simultaneous validity.  Put
$L=\floor{\log_p(AU)}$.  Since $1\le A\le U^C$, one has
$L\asymp\log U$ uniformly.  Fix a small $0<\delta<1/8$ and let
\[
 \mathcal J=\{j\in\mathbb Z:\delta L\le j\le(1-\delta)L\}.
\]
Set
\[
 S_{\mathcal J}(u)=\sum_{j\in\mathcal J}
 \varepsilon_{p,j}(Au+b).
\]
In this proof, $\E$ and $\Var$ refer to the uniform probability space
on $I$.  Apply \cref{lem:digitcovariance} with a large fixed
exponent $M_*$.  Summing
\eqref{eq:digitmean} over $j\in\mathcal J$ gives
\begin{equation}\label{eq:sunitmean}
 \E S_{\mathcal J}
 =\frac{p-1}{2}|\mathcal J|
 +O_p\!\left(L(\log U)^{-M_*}\right).
\end{equation}
For the variance, the diagonal terms contribute $O_p(L)$.  The number
of ordered off-diagonal pairs satisfying
\[
 |j_1-j_2|<G\log\log U
\]
is $O(L\log\log U)$, and each such covariance is $O_p(1)$.  Every
remaining pair satisfies \eqref{eq:digitcovariance}.  Consequently,
\begin{equation}\label{eq:sunitvariance}
 \Var(S_{\mathcal J})
 \ll_p L\log\log U+
 L^2(\log U)^{-M_*}=o(L^2)
\end{equation}
uniformly in all admissible $A,b,I$.

To make the uniform exceptional function explicit, let
\[
 \omega(U)=\sup_{A,b,I}
 \left(
 \frac{\Var(S_{\mathcal J})}{L^2}
 +\frac{1}{(\log U)^{M_*}}
 \right)^{1/4},
\]
where the supremum is over the displayed admissible data.  Enlarge
$\omega$ monotonically if necessary; then $\omega(U)\to0$ by
\eqref{eq:sunitvariance}.  Chebyshev's inequality and
\eqref{eq:sunitmean} show that, outside at most
$O(\omega(U)^2U)$ values of $u$,
\[
 S_{\mathcal J}(u)
 =\frac{p-1}{2}|\mathcal J|+O(\omega(U)L)+o(L).
\]
Because $u\asymp U$ and $|b|$ is polylogarithmic, one has
$Au+b\asymp AU$ uniformly; hence the full expansion has $L+O(1)$
possible digit positions.  Write its omitted contribution as
$R_{\mathcal J}(u)$.  Then
\[
 0\le R_{\mathcal J}(u)
 \le 2\delta(p-1)L+O_p(1),
 \qquad
 |\mathcal J|=(1-2\delta)L+O(1).
\]
Consequently the preceding central estimate gives, uniformly outside
$o(U)$ values,
\[
 -\delta(p-1)L+o(L)
 \le s_p(Au+b)-\frac{p-1}{2}L
 \le \delta(p-1)L+o(L).
\]
First choose $\delta$ sufficiently small in terms of $p$ and $\eps$,
and then let $U\to\infty$.  This gives
\[
 \left|s_p(Au+b)-\frac{p-1}{2}L\right|
 \le \frac\eps2(p-1)L
\]
for all but $o(U)$ values.  Since
$L=\log_p(AU)+O(1)$, the bounded change is absorbed by the remaining
$\eps$ margin.  A finite union over $p\in\mathcal P$ completes the
proof.  More precisely, the union has size at most $r(U)U$, where
$r(U)=O_{\mathcal P}(\omega(U)^2)\to0$ and depends only on the fixed
data in the theorem.  This is the asserted uniform meaning of
``all but $o(U)$''.
\end{proof}

We retain the exact same-base identity used in the small-prime
budgets.

\begin{lemma}[Forced suffix]\label{lem:forcedsuffix}
For every prime $p$, integer $J\ge0$, and integer $t\ge1$,
\begin{equation}\label{eq:forcedsuffix}
 s_p(p^Jt-1)=(p-1)J+s_p(t-1).
\end{equation}
For the asymptotic consequence, fix a prime $p$ and $\delta,\eps>0$.  Suppose that
\[
 0\le e\le1-\delta,
 \qquad
 p^J=X^{e+O(\log\log X/\log X)},
 \qquad
 p^JU\asymp X,
\]
and that $t$ ranges through an interval of length $\asymp U$ contained
in a fixed proportional interval at scale $U$.  Then, uniformly in the
displayed data, all but $o(U)$ such $t$ satisfy
\begin{equation}\label{eq:forcedsuffixconsequence}
 s_p(p^Jt-1)
 \ge\left(\frac{1+e}{2}-\eps\right)(p-1)\log_pX.
\end{equation}
\end{lemma}

\begin{proof}
The identity follows from
\[
 p^Jt-1=p^J(t-1)+(p^J-1),
\]
whose summands occupy disjoint digit blocks.  The scale assumptions give
\[
 \log U=(1-e)\log X+O(\log\log X)
 \ge \frac\delta2\log X
\]
for all sufficiently large $X$.  Thus $U\ge X^{\delta/2}$, and
\cref{lem:intervalconcentration} applies to the interval of values
$t-1$.  Outside $o(U)$ values it gives
\[
 s_p(t-1)\ge
 \left(\frac12-\frac\eps3\right)(p-1)\log_pU.
\]
Also
\[
 J=e\log_pX+O\!\left(\frac{\log\log X}{\log p}\right),
 \qquad
 \log_pU=(1-e)\log_pX+O\!\left(\frac{\log\log X}{\log p}\right).
\]
Substitution in \eqref{eq:forcedsuffix}, followed by absorption of the
$O(\log\log X/\log p)$ loss into the fixed $\eps\log_pX$ margin,
proves \eqref{eq:forcedsuffixconsequence}.
\end{proof}

\section{Abundant exact representations and the large-prime sieve}

We first make the representation count uniform, including the
geometric conditions needed for the factorial arguments.

\begin{lemma}[Uniform abundant representations]
\label{lem:representations}
Fix an integer $r\ge2$, numbers $0<\alpha<\beta$, and exponents
$0<e_i<1$.  Let
\[
 M_i=X^{e_i+O(\log\log X/\log X)},
 \qquad U_i=\frac{X}{M_i},
\]
where the implicit constants are fixed, and suppose that there are
distinct indices $i_0,i_1$ for which
\[
 \gcd(M_{i_0},M_{i_1})=1,
 \qquad M_{i_0}M_{i_1}=o(X).
\]
There is a finite cover of $[\alpha X,\beta X]$ by target intervals
$\mathcal T_\ell$ of length $\asymp X$ with the following property.
For each $\ell$ there are fixed intervals $I_{i,\ell}$ of consecutive
positive integers, independent of the individual target
$m\in\mathcal T_\ell$, such that
\[
 c_iU_i\le |I_{i,\ell}|\le C_iU_i
\]
for positive constants $c_i,C_i$, and constants
$0<c_R<C_R<\infty$ for which every $m\in\mathcal T_\ell$ has between
\begin{equation}\label{eq:representationtwosided}
 c_R\frac{\prod_{i=1}^rU_i}{X}
 \quad\text{and}\quad
 C_R\frac{\prod_{i=1}^rU_i}{X}
\end{equation}
representations
\begin{equation}\label{eq:representationeq}
 M_1t_1+\cdots+M_rt_r=m,
 \qquad t_i\in I_{i,\ell}.
\end{equation}
Moreover, the intervals can be chosen so that there are fixed
constants $0<\gamma_0<\Gamma_0<\infty$ for which
\begin{equation}\label{eq:representationbox}
 \gamma_0X\le M_it\le\Gamma_0X
 \qquad(t\in I_{i,\ell},\ 1\le i\le r).
\end{equation}
Thus the first comparison holds on the entire Cartesian parameter box,
not merely on represented tuples.  In addition, every representation
counted above satisfies
\begin{equation}\label{eq:representationgeometry}
 M_it_i\le(1-\gamma_0)m
 \qquad(1\le i\le r),
\end{equation}
after decreasing $\gamma_0$ once if necessary.
\end{lemma}

\begin{proof}
Relabel the coordinates, if necessary, so that the distinguished
coprime pair is $(M_1,M_2)$.  All conclusions of the lemma are
symmetric in the coordinates.  Fix a center $\tau\in(\alpha,\beta)$.
Choose positive numbers
$\theta_1,\ldots,\theta_r$ with sum $\tau$; for example,
$\theta_i=\tau/r$.  Choose $\rho>0$ so small that
\[
 0<\theta_i-3\rho<\theta_i+3\rho<\tau-3\rho
 \qquad(1\le i\le r).
\]
Let
\[
 I_i=\{t\in\Z:\ (\theta_i-\rho)U_i\le t
 \le(\theta_i+\rho)U_i\}.
\]
These are fixed for the entire target interval
\[
 \mathcal T_\tau=[(\tau-\rho/4)X,(\tau+\rho/4)X].
\]
For large $X$, all their elements are positive and $|I_i|\asymp U_i$.
For every $t\in I_i$,
\[
 (\theta_i-\rho)X-O(M_i)\le M_it
 \le(\theta_i+\rho)X+O(M_i).
\]
Since $e_i<1$, one has $M_i=o(X)$.  Hence, after enlarging $X_0$, there
are fixed constants $0<\gamma_0<\Gamma_0<\infty$ such that
\eqref{eq:representationbox} holds for every point of every parameter
interval.  Moreover, for $m\in\mathcal T_\tau$ the inequalities
$\theta_i+3\rho<\tau-3\rho$ give a fixed positive proportional gap
between $M_it$ and $m$.  Decreasing $\gamma_0$ if necessary yields
\eqref{eq:representationgeometry} for every counted representation.

Restrict $t_3,\ldots,t_r$ to the concentric intervals obtained by
replacing $\rho$ by $\rho/(10r)$.  There are
\begin{equation}\label{eq:innerchoices}
 \asymp\prod_{i=3}^rU_i
\end{equation}
such choices, with the empty product interpreted as $1$.  For every
$m\in\mathcal T_\tau$, the residual
\[
 m'=m-\sum_{i=3}^rM_it_i
\]
lies in a fixed inner subinterval of the real interval
$M_1I_1+M_2I_2$.  Indeed, uniformly in the restricted choices,
\[
 \left|m'-(\theta_1+\theta_2)X\right|
 \le\left(\frac14+\frac{r-2}{10r}\right)\rho X+o(X)
 \le\frac{7\rho}{20}X+o(X),
\]
whereas the two-variable sum interval has center
$(\theta_1+\theta_2)X+o(X)$ and half-length $2\rho X+o(X)$.
Thus, for all large $X$, the residual is at distance at least
$\rho X$ from each endpoint.  This is precisely why the same $I_i$
work for every target in $\mathcal T_\tau$.

Since $\gcd(M_1,M_2)=1$, the integer solutions of
$M_1t_1+M_2t_2=m'$ form a lattice line
\[
 t_1=t_1^{(0)}+M_2z,
 \qquad
 t_2=t_2^{(0)}-M_1z,
 \qquad z\in\Z.
\]
Write $I_i=[A_i,B_i]\cap\mathbb Z$.  The two restrictions on $z$ are
\[
 \frac{A_1-t_1^{(0)}}{M_2}\le z\le
 \frac{B_1-t_1^{(0)}}{M_2},
 \qquad
 \frac{t_2^{(0)}-B_2}{M_1}\le z\le
 \frac{t_2^{(0)}-A_2}{M_1}.
\]
Each interval has length $\asymp X/(M_1M_2)$.  More explicitly,
after translating both $z$-intervals by $t_1^{(0)}/M_2$ and multiplying
by $M_1M_2$, their intersection becomes
\[
 [M_1A_1,M_1B_1]\cap[m'-M_2B_2,m'-M_2A_2].
\]
The preceding inner-position choice gives one fixed $\xi>0$ such that
\begin{align*}
 m'-(M_1A_1+M_2A_2)&\ge\xi X,\\
 (M_1B_1+M_2B_2)-m'&\ge\xi X,
\end{align*}
and the two weighted interval lengths are each at least $\xi X$ and
at most $\xi^{-1}X$.  For any two intervals, the length of their
intersection is at least the minimum of their two lengths and the two
endpoint margins above.  Hence the original $z$-interval intersection
has length between
\[
 c_z\frac{X}{M_1M_2}
 \quad\text{and}\quad
 C_z\frac{X}{M_1M_2}
\]
for fixed $0<c_z<C_z<\infty$.  Since $M_1M_2=o(X)$, the number of
integer $z$ is bounded above and below by fixed positive multiples of
$X/(M_1M_2)$.  Multiplication by
\eqref{eq:innerchoices} yields
\[
 \asymp \frac{X}{M_1M_2}\prod_{i=3}^rU_i
 =\frac1X\prod_{i=1}^rU_i,
\]
with uniform lower and upper constants.

For the global upper bound, do not restrict the last $r-2$
coordinates.  For each of their $O(\prod_{i=3}^rU_i)$ choices, the
same lattice-line calculation gives $O(X/(M_1M_2))$ solutions.  This
proves the upper half of \eqref{eq:representationtwosided}.  Finally,
a finite set of centers $\tau$ covers the compact interval
$[\alpha,\beta]$, proving the lemma.
\end{proof}

\begin{corollary}[Representation averaging]
\label{cor:representationaveraging}
Put $V=\prod_iU_i$ and $R=V/X$.  In one target interval supplied by
\cref{lem:representations}, if $o(V)$ representation tuples are
bad, then only $o(X)$ targets can have every one of their
representations bad.
\end{corollary}

\begin{proof}
If $T$ targets had all representations bad, the lower bound in
\eqref{eq:representationtwosided} would give at least $c_RTR$ bad
tuples.  Since this is $o(V)$ and $R=V/X$, one has $T=o(X)$.
\end{proof}

We now handle all primes above a polylogarithmic threshold.  The carry
propagation is written column by column.

\begin{lemma}[Large-prime Kummer sieve]\label{lem:largeprime}
Fix $r\ge2$ and $\delta>0$.  Suppose every $M_i$ has all prime factors
in one fixed finite set, each parameter interval has length
$U_i\ge X^\delta$, and
\[
 a_i=M_it_i-1\asymp X.
\]
Let $v=O(\log X)$ and set
\[
 n=\sum_{i=1}^ra_i-v,
 \qquad N=n+v=\sum_{i=1}^ra_i.
\]
For any fixed $B\ge4$, put $Y=(\log X)^B$.  Then the proportion of
parameter tuples for which some prime $p>Y$ violates
\begin{equation}\label{eq:largeprimegoal}
 \nu_p((n+1)\cdots(n+v))
 \le \nu_p\binom{N}{a_1,\ldots,a_r}
\end{equation}
is $o(1)$.
\end{lemma}

\begin{proof}
For large $X$, $Y>v$ and $Y$ exceeds every prime dividing any $M_i$.
If a prime $p>Y$ divides the rising product, there is a unique
$1\le j\le v$ for which $p\mid n+j$.  Put
$e=\nu_p(n+j)$.  The demand on the left side of
\eqref{eq:largeprimegoal} is exactly $e$.  Since
\[
 N=(n+j)+(v-j),\qquad 0\le v-j<p,
\]
the first $e$ target digits of $N$ in base $p$ are
\[
 \boxed{\ v-j\ \big|\ 0\ \big|\ 0\ \big|\ \cdots\ \big|\ 0\ },
\]
with $e-1$ zeros after the units digit.

Let $c_\ell$ be the carry entering column $\ell$ in the addition of
$a_1,\ldots,a_r$, so $c_0=0$.  In the first $e$ columns the carry
recursion is
\begin{align*}
 \sum_{i=1}^r a_{i,0}&=v-j+pc_1,\\
 c_\ell+\sum_{i=1}^r a_{i,\ell}&=pc_{\ell+1}
 \qquad(1\le\ell\le e-1).
\end{align*}
If $c_1>0$, then the second line recursively forces
$c_2,\ldots,c_e\ge1$.  Hence the total Kummer carry mass is at least
$c_1+\cdots+c_e\ge e$, and \eqref{eq:largeprimegoal} holds.  Failure
therefore requires $c_1=0$.

Let $h_i$ be the least nonnegative residue of $a_i$ modulo $p$.  The
condition $c_1=0$ gives
\begin{equation}\label{eq:residuevector}
 h_1+\cdots+h_r=v-j.
\end{equation}
The number of possible residue vectors as $j$ varies is bounded
explicitly by
\begin{equation}\label{eq:compositioncount}
 \sum_{s=0}^{v-1}\binom{s+r-1}{r-1}
 =\binom{v+r-1}{r}=O_r(v^r).
\end{equation}

Fix $t_1$ and one vector $(h_1,\ldots,h_r)$.  By
\eqref{eq:residuevector}, every $h_i\le v$, whereas $a_i\asymp X$.
Thus, for large $X$, the integer $a_1-h_1$ is positive and
$\asymp X$ (in particular, the case $p>a_1$ cannot occur under
\eqref{eq:residuevector}).  Every candidate prime $p>Y$ divides it, so
there are at most
\[
 \frac{\log(CX)}{\log Y}=O\!\left(\frac{\log X}{\log Y}\right)
\]
such primes.  For a candidate $p$ and each $i\ge2$, the congruence
\[
 M_it_i-1\equiv h_i\pmod p
\]
has at most $U_i/p+1$ solutions in the interval for $t_i$, because
$p\nmid M_i$.  Dividing by the full parameter-box size and using
\eqref{eq:compositioncount}, the bad proportion is
\begin{equation}\label{eq:largeprimebound}
 \ll_r v^r\frac{\log X}{\log Y}
 \prod_{i=2}^r\left(\frac1Y+\frac1{U_i}\right).
\end{equation}
For $r=2$, which is the worst case, the part containing $1/Y$ is
\[
 O\!\left(
 \frac{(\log X)^3}{Y\log\log X}
 \right)
 =O\!\left(\frac{(\log X)^{3-B}}{\log\log X}\right)=o(1)
\]
when $B\ge4$, and the $U_2^{-1}$ term is
$O((\log X)^3/(X^\delta\log\log X))=o(1)$.  For completeness, expand the product for arbitrary fixed $r$:
\[
 \prod_{i=2}^r\left(\frac1Y+\frac1{U_i}\right)
 =\sum_{S\subseteq\{2,\ldots,r\}}
 Y^{-(r-1-|S|)}\prod_{i\in S}U_i^{-1}.
\]
The empty-set term in \eqref{eq:largeprimebound} is
\[
 O_r\!\left(\frac{(\log X)^{r+1-B(r-1)}}{\log\log X}\right)=o(1)
\]
for $B\ge4$, because $r+1-4(r-1)=5-3r<0$ for $r\ge2$.  Every term
with $|S|\ge1$ contains the factor $X^{-\delta|S|}$ and is therefore
$o(1)$ regardless of the remaining fixed power of $\log X$.  This
proves the result for every fixed $B\ge4$.
\end{proof}

\section{Resource allocation}

Put $h=k-1$.  The next elementary lemma converts two continuous logarithmic resources into exactly $k$ pure-power parts while retaining a coprime pair whose product is $o(X)$.

\begin{lemma}[Pure-power allocation]\label{lem:allocation}
Let $k\ge2$, and suppose that $\lambda_2,\lambda_3>0$ and
\[
 \lambda_2+\lambda_3<k-1.
\]
Then there are integers $r,s\ge1$ with $r+s=k$, and exponents
\[
 0<x_1,\ldots,x_r<1,
 \qquad
 0<y_1,\ldots,y_s<1,
\]
such that
\[
 \sum_{i=1}^rx_i=\lambda_2,
 \qquad
 \sum_{j=1}^sy_j=\lambda_3,
\]
and, after relabeling,
\[
 x_1+y_1<1.
\]
If $F\subset(0,1)$ is finite and $\lambda_2,\lambda_3\notin F$, the exponents may also be chosen outside $F$.
\end{lemma}

\begin{proof}
Start with
\[
 r_0=\floor{\lambda_2}+1,
 \qquad
 s_0=\floor{\lambda_3}+1.
\]
Since $\lambda_2+\lambda_3<k-1$, one has $r_0+s_0\le k$.  Choose $r\ge r_0$ and $s\ge s_0$ with $r+s=k$.

If $r_0+s_0<k$, assign at least one extra slot to, say, the binary group.  Then $r-1>\lambda_2$.  One binary exponent can be made arbitrarily small while the remaining $r-1$ positive exponents, all below $1$, sum to the remainder.  Since every ternary exponent is below $1$, the small binary exponent and the least ternary exponent have sum below $1$.

If $r_0+s_0=k$, choose $r-1$ binary exponents and $s-1$ ternary exponents sufficiently close to $1$ from below.  The two remaining exponents then approach
\[
 \lambda_2-(r-1)
 \qquad\text{and}\qquad
 \lambda_3-(s-1),
\]
whose sum is
\[
 \lambda_2+\lambda_3-(k-2)<1.
\]
A sufficiently small perturbation makes every exponent positive and below $1$ and preserves the strict inequality.

In either construction, every group containing at least two exponents has a nonempty relatively open family of admissible splittings of its prescribed sum.  The conditions that one of those exponents belong to the finite set $F$ form a finite union of proper affine subspaces, so the splitting may be chosen outside them.  A group consisting of a single exponent is fixed at $\lambda_2$ or $\lambda_3$, which lies outside $F$ by hypothesis.
\end{proof}

\begin{lemma}[Resource inequalities]\label{lem:resource}
Fix
\[
 c<\frac{3h}{\log 12}.
\]
There exist a number $c_0$ with
\[
 c<c_0<\frac{3h}{\log 12}
\]
and positive numbers $\lambda_2,\lambda_3$ such that
\begin{equation}\label{eq:resource2}
 \frac{h+\lambda_2}{2\log2}>c_0,
\end{equation}
\begin{equation}\label{eq:resource3}
 \frac{h+\lambda_3}{\log3}>c_0,
\end{equation}
and
\begin{equation}\label{eq:resourcesum}
 \lambda_2+\lambda_3<h.
\end{equation}
Moreover, $\lambda_2$ and $\lambda_3$ may be required to avoid any prescribed finite subset of $(0,\infty)$.
\end{lemma}

\begin{proof}
Choose $c_0$ sufficiently close to $3h/\log12$, while keeping $c_0>c$, and then choose a small $\rho>0$.  Put
\[
 \lambda_2=2c_0\log2-h+\rho,
 \qquad
 \lambda_3=c_0\log3-h+\rho.
\]
At the endpoint $3h/\log12$, both limiting values are positive, so $\lambda_2$ and $\lambda_3$ are positive when $c_0$ is sufficiently close to the endpoint.  The inequalities \eqref{eq:resource2} and \eqref{eq:resource3} are immediate.  Finally,
\[
 \lambda_2+\lambda_3=c_0\log12-2h+2\rho<h
\]
when $\rho$ is sufficiently small, because $c_0\log12<3h$.  The admissible values of $\rho$ contain an interval; deleting the finitely many values that would place $\lambda_2$ or $\lambda_3$ in a prescribed finite set proves the final assertion.
\end{proof}

\section{Proof of the density-one lower bound}

\begin{proof}[Proof of \cref{thm:main}]
Put $h=k-1$.  The assertion is trivial for $c\le0$, so fix
\[
 0<c<\frac{3h}{\log12}.
\]
We first choose every fixed parameter, in an order that will remain in
force throughout the proof.

\smallskip
\noindent\emph{Parameter hierarchy.}
Use \cref{lem:resource} to choose
\[
 c<c_1<c_0<\frac{3h}{\log12}
\]
and $\lambda_2,\lambda_3>0$ such that
\[
 \frac{h+\lambda_2}{2\log2}>c_0,
 \qquad
 \frac{h+\lambda_3}{\log3}>c_0,
 \qquad
 \lambda_2+\lambda_3<h,
\]
with $\lambda_2,\lambda_3\ne1/2$.  Define the positive margins
\begin{equation}\label{eq:margins}
 \delta_2=\frac{h+\lambda_2}{2\log2}-c_1,
 \qquad
 \delta_3=\frac{h+\lambda_3}{\log3}-c_1,
 \qquad
 \delta_5=\frac{2h}{\log5}-c_1,
\end{equation}
and the representation surplus
\begin{equation}\label{eq:resourcegap}
 \Delta=h-\lambda_2-\lambda_3>0.
\end{equation}
The last margin in \eqref{eq:margins} is positive because
$3h/\log12<2h/\log5$.

Apply \cref{lem:allocation} with $F=\{1/2\}$.  We obtain $r,s\ge1$,
$r+s=k$, and exponents
\[
 x_1,\ldots,x_r,\quad y_1,\ldots,y_s\in(0,1)\setminus\{1/2\}
\]
such that
\[
 \sum_{i=1}^rx_i=\lambda_2,
 \qquad
 \sum_{j=1}^sy_j=\lambda_3,
 \qquad
 x_1+y_1<1.
\]
Let $e_1,\ldots,e_k$ denote these exponents in one list and put
\begin{equation}\label{eq:exponentgap}
 \delta_{\rm exp}=\frac14\min_{1\le i\le k}
 \{e_i,1-e_i,|1-2e_i|\}>0.
\end{equation}
Choose a finite-prime error $\rho>0$ so small that
\begin{equation}\label{eq:rhochoice}
 (k+1)\rho<\frac14\min\{\delta_2,\delta_3,\delta_5\}.
\end{equation}
Choose the medium-prime digit error $\eta>0$ so that
\begin{equation}\label{eq:etachoice}
 \eta<\frac{\Delta}{16(k+1)}.
\end{equation}
Next choose a fixed prime cutoff $P\ge5$ so large that
\begin{equation}\label{eq:Pchoice}
 c_1\frac{\log p}{p-1}<\frac{\Delta}{3}
 \qquad(p>P).
\end{equation}
Finally choose an integer $B\ge4$, put $D=B+4$, and later take $X$
large enough for every estimate made with these fixed parameters.
This order of choices eliminates any diagonal limiting convention.

\smallskip
\noindent\emph{The exact representations.}
Work on $X\le n<2X$ and set
\[
 v=\floor{c_1\log X}.
\]
For $1\le i\le r$ and $1\le j\le s$, define
\[
 J_i=\floor{\frac{x_i\log X}{\log2}},
 \qquad M_i=2^{J_i},
\]
\[
 K_j=\floor{\frac{y_j\log X}{\log3}},
 \qquad M_{r+j}=3^{K_j}.
\]
Thus $M_i\asymp X^{e_i}$, with absolute multiplicative constants, and
\[
 \gcd(M_1,M_{r+1})=1,
 \qquad
 M_1M_{r+1}\ll X^{x_1+y_1}=o(X).
\]
Put $U_i=X/M_i$.  The target
\[
 m=n+v+k
\]
lies in $[X/2,3X]$ for all sufficiently large $X$.  The pair
$(M_1,M_{r+1})$ supplies the distinguished coprime pair required by
\cref{lem:representations}.  Apply that lemma to this fixed
proportional target interval.  After a finite subdivision, each target subinterval has
fixed parameter intervals $I_i$ of length $\asymp U_i$ such that every
$m$ in it has between fixed positive multiples of
\begin{equation}\label{eq:Rcount}
 R=\frac1X\prod_{i=1}^kU_i
 \asymp X^{k-1-(\lambda_2+\lambda_3)}=X^\Delta
\end{equation}
representations
\begin{equation}\label{eq:targetrepresentation}
 \sum_{i=1}^kM_it_i=n+v+k,
 \qquad t_i\in I_i.
\end{equation}
For every point of the full Cartesian parameter box, define
\[
 a_i=M_it_i-1.
\]
The strengthened full-box conclusion \eqref{eq:representationbox}
gives fixed constants $0<\gamma_0<\Gamma_0$ such that
\begin{equation}\label{eq:fullboxparts}
 \frac{\gamma_0}{2}X\le a_i\le2\Gamma_0X
 \qquad(t_i\in I_i,\ 1\le i\le k)
\end{equation}
for all sufficiently large $X$.  This is the comparison used below
when digit estimates and the large-prime sieve are applied to the
entire parameter box.

For a tuple satisfying \eqref{eq:targetrepresentation}, one has
\[
 a_1+\cdots+a_k=n+v.
\]
The represented-tuple conclusion \eqref{eq:representationgeometry},
together with $m=n+O(\log X)$, gives a fixed $\gamma>0$ for which
\begin{equation}\label{eq:partsinterior}
 \gamma X\le a_i\le(1-\gamma)n
 \qquad(1\le i\le k)
\end{equation}
for all large $X$.  In particular every $t_i$ and every represented
factorial argument $a_i$ is positive.

The exponent ranges required by \cref{thm:sunitnormal} are also
explicit.  For a binary exponent $x_i$,
\[
 J_i=\frac{x_i}{\log2}\log X+O(1),
 \qquad
 \log U_i=(1-x_i)\log X+O(1),
\]
so, for large $X$,
\begin{equation}\label{eq:Jbound}
 J_i\le C_i\log U_i,
 \qquad C_i=\frac{2x_i}{(1-x_i)\log2}.
\end{equation}
Likewise, for a ternary exponent $y_j$,
\begin{equation}\label{eq:Kbound}
 K_j\le C'_j\log U_{r+j},
 \qquad C'_j=\frac{2y_j}{(1-y_j)\log3}.
\end{equation}
Thus every $S$-unit normal-order invocation below lies in the proved
uniform range.

Let
\[
 V=\prod_{i=1}^k|I_i|\asymp\prod_{i=1}^kU_i.
\]
We shall remove only $o(V)$ representation tuples.  A coordinatewise
exceptional set of size $o(U_i)$ removes $o(V)$ tuples after the other
coordinates are restored, and a finite union preserves this property.

Because the exponents $e_i$ are fixed in $(0,1)$, there is one constant
$C_*$ such that
\[
 M_i\le U_i^{C_*}\qquad(1\le i\le k)
\]
for all large $X$.  Every interval $I_i$ is contained in a fixed
proportional interval at scale $U_i$, and $M_it_i-1>0$.  Hence all
applications of \cref{thm:sunitnormal} below are uniform in the
required multipliers, shifts, and intervals.  In addition,
\eqref{eq:exponentgap} gives $1-e_i\ge4\delta_{\rm exp}$ for every
resource exponent.  Since each same-base multiplier satisfies
$p^J=X^{e_i+O(1/\log X)}$, one has $p^JU_i=X$ and the interval $I_i$
is proportional at scale $U_i$.  Thus every use of
\cref{lem:forcedsuffix} below satisfies its missing-growth hypothesis,
for example with $\delta=2\delta_{\rm exp}$.

\smallskip
\noindent\emph{The prime $2$.}
For a binary part, \cref{lem:forcedsuffix}, invoked with error small
enough that the natural-log coefficient error is at most $\rho$, gives
\begin{equation}\label{eq:binarypart2}
 s_2(2^{J_i}t_i-1)
 \ge\left(\frac{1+x_i}{2\log2}-\rho\right)\log X
\end{equation}
outside $o(U_i)$ values.  For a ternary part,
\cref{thm:sunitnormal}, with $S=\{3\}$, $\mathcal P=\{2\}$,
$A=3^{K_j}$, $b=-1$, and \eqref{eq:Kbound}, gives
\begin{equation}\label{eq:ternarypart2}
 s_2(3^{K_j}t_{r+j}-1)
 \ge\left(\frac1{2\log2}-\rho\right)\log X
\end{equation}
outside $o(U_{r+j})$ values.  By \cref{lem:targetupper}, all but
$o(X)$ targets satisfy
\begin{equation}\label{eq:target2}
 s_2(n)\le\left(\frac1{2\log2}+\rho\right)\log X.
\end{equation}
For such a target and a tuple surviving the coordinate exceptions,
\begin{align}
 \sum_{i=1}^ks_2(a_i)-s_2(n)
 &\ge\left(\frac{h+\lambda_2}{2\log2}-(k+1)\rho\right)\log X\notag\\
 &>(c_1+\tfrac34\delta_2)\log X>v.            \label{eq:prime2budget}
\end{align}

\smallskip
\noindent\emph{The prime $3$.}
The same argument, using \cref{thm:sunitnormal} with the roles of
$2$ and $3$ interchanged, gives
for ternary parts
\[
 s_3(3^{K_j}t_{r+j}-1)
 \ge\left(\frac{1+y_j}{\log3}-\rho\right)\log X,
\]
for binary parts
\[
 s_3(2^{J_i}t_i-1)
 \ge\left(\frac1{\log3}-\rho\right)\log X,
\]
and for normal targets
\[
 s_3(n)\le\left(\frac1{\log3}+\rho\right)\log X.
\]
Consequently,
\begin{equation}\label{eq:prime3budget}
 \sum_{i=1}^ks_3(a_i)-s_3(n)
 \ge\left(\frac{h+\lambda_3}{\log3}-(k+1)\rho\right)\log X
 >(c_1+\tfrac34\delta_3)\log X>v.
\end{equation}

\smallskip
\noindent\emph{The fixed primes $5\le p\le P$.}
For this finite set, every multiplier $M_i$ is a power of a prime
different from $p$.  Use \cref{thm:sunitnormal} with $S=\{2,3\}$ and
$\mathcal P=\{p:5\le p\le P\}$, choosing its error so that the
resulting natural-log error is at most $\rho\log X$ per part.  Together with the
target upper estimate, this gives
\begin{align}
 \sum_{i=1}^ks_p(a_i)-s_p(n)
 &\ge\left(\frac h2\frac{p-1}{\log p}-(k+1)\rho\right)\log X\notag\\
 &\ge\left(\frac{2h}{\log5}-(k+1)\rho\right)\log X
 >(c_1+\tfrac34\delta_5)\log X>v,             \label{eq:fixedpbudget}
\end{align}
because $(p-1)/\log p$ is increasing for $p\ge3$.

\smallskip
\noindent\emph{The medium primes $P<p\le Y$.}
Put $Y=(\log X)^B$ and $u_i=1-e_i$.  Then
\begin{equation}\label{eq:usum}
 \sum_{i=1}^ku_i=k-(\lambda_2+\lambda_3)=1+\Delta.
\end{equation}
The conditions of \cref{lem:twoblock} hold uniformly for every part:
by \eqref{eq:exponentgap}, the two scale exponents are separated from
$0$, $1$, and one another; the parameter intervals have length
$\asymp U_i$; the full-box estimate \eqref{eq:fullboxparts} gives
$M_it_i-1\asymp X$ for every $t_i\in I_i$; and $d=-1$ is allowed
explicitly.  Moreover $p\nmid M_i$ because $p>P\ge5$.
Apply that lemma with $A_0=B$, error $\eta$, and $D=B+4$.  Its
quantitative exceptional bound, summed over all $p\le Y$ and all $k$
coordinates, removes $o(V)$ tuples.  Every remaining tuple satisfies
simultaneously
\begin{equation}\label{eq:mediumindividual}
 \frac{s_p(a_i)}{p-1}
 \ge\left(\min\left\{u_i,\frac12\right\}-\eta\right)\log_pX.
\end{equation}
The target upper lemma, also applied with $A=B$ and error $\eta$,
gives
\[
 \frac{s_p(n)}{p-1}\le\left(\frac12+\eta\right)\log_pX
\]
for every $p\le Y$, outside $o(X)$ targets.  Since
$\min\{u_i,1/2\}\ge u_i/2$, \eqref{eq:usum} and
\eqref{eq:etachoice} imply
\begin{align}
 \frac{\sum_i s_p(a_i)-s_p(n)}{p-1}
 &\ge\left(\frac\Delta2-(k+1)\eta\right)\log_pX\notag\\
 &>\frac{7\Delta}{16}\log_pX
 >\frac\Delta3\log_pX.                       \label{eq:mediumbudget}
\end{align}
By \eqref{eq:Pchoice}, the last quantity is greater than
$c_1\log X/(p-1)\ge v/(p-1)$.  Hence the digit criterion holds
simultaneously for every medium prime.

\smallskip
\noindent\emph{Large primes and averaging.}
All multipliers have prime factors in $\{2,3\}$, and
\[
 U_i\gg X^{\min_j(1-e_j)}.
\]
The hypothesis $a_i\asymp X$ in \cref{lem:largeprime} holds throughout
the full Cartesian box by \eqref{eq:fullboxparts}, rather than only on
represented tuples.  Thus that lemma, with the already fixed $B\ge4$,
removes only $o(V)$ tuples and proves the required rising-product
valuation inequality for every $p>Y$.

Intersect the normal-target sets used above.  Their complement has
size $o(X)$.  By the uniform upper bound in
\cref{lem:representations}, all representations of abnormal targets
account for at most
\[
 o(X)\cdot O(V/X)=o(V)
\]
tuples.  The fixed-prime coordinate exceptions, the simultaneous
two-block exceptions, and the large-prime sieve exceptions together
also total $o(V)$.  Hence only $o(V)$ representations have been
removed.

By \cref{cor:representationaveraging}, all but $o(X)$ targets in each
fixed target subinterval retain a representation satisfying the digit
criterion for every $p\le Y$ and the equivalent Kummer valuation
inequality for every $p>Y$.  Lemma~\ref{lem:digitcriterion} therefore
gives
\[
 a_1!\cdots a_k!\mid n!,
 \qquad
 a_1+\cdots+a_k-n=v.
\]
Since $n<2X$ and $c_1>c$,
\[
 v-c\log n\ge(c_1-c)\log X-c\log2-1>0
\]
for large $X$.  Thus $g_k(n)\ge c\log n$ outside $o(X)$ integers in
$[X,2X)$.

It remains only to justify the passage from dyadic intervals to
natural density.  Given $\eps>0$, choose a dyadic index $m_0$ so large
that the exceptional count in every block
$[2^m,2^{m+1})$, $m\ge m_0$, is at most $\eps2^m$.  Summing the
geometric series shows that the exceptions up to any $x$ contribute
at most $2\eps x+O(2^{m_0})$.  Divide by $x$ and then let
$\eps\downarrow0$.  This proves the global density-one assertion.
\end{proof}

\section{Consequences, comparisons, and further directions}

\begin{proof}[Proof of \cref{cor:summatory}]
Fix $c<3(k-1)/\log12$ and define
\[
 \mathcal G_c=\{n\in\N:g_k(n)\ge c\log n\}.
\]
By \cref{thm:main}, $|[1,x]\setminus\mathcal G_c|=o(x)$.  Also
$g_k(n)\ge0$ for every $n$: if $n\ge k$, take
$(n-k+1,1,\ldots,1)$; if $n<k$, take $(1,\ldots,1)$, whose factorial
product is $1$ and whose excess is $k-n>0$.  Therefore
\[
 \sum_{n\le x}g_k(n)
 \ge c\sum_{\substack{n\le x\\n\in\mathcal G_c}}\log n
 =c x\log x+o(x\log x).
\]
Let $c\uparrow3(k-1)/\log12$.  The upper summatory bound follows by
summing \cref{thm:upper}; the $\log\log n$ term and bounded error are
$o(x\log x)$.
\end{proof}

The four prime ranges and their logically separate inputs are:
\begin{center}
\small
\begin{tabular}{@{}l>{\raggedright\arraybackslash}p{0.34\textwidth}>{\raggedright\arraybackslash}p{0.31\textwidth}@{}}
\toprule
Prime range & Input & Purpose \\
\midrule
$p=2,3$ & forced suffixes and \cref{thm:sunitnormal} & endpoint optimization \\
fixed $5\le p\le P$ & \cref{thm:sunitnormal} & full normal-order digit mass \\
$P<p\le(\log X)^B$ & quantitative two-block estimate & simultaneous growing-base control \\
$p>(\log X)^B$ & explicit Kummer carry sieve & isolated large prime-power demands \\
\bottomrule
\end{tabular}
\end{center}

At formal equality, the binary and ternary constraints give
\[
 \lambda_2=2c\log2-h,
 \qquad \lambda_3=c\log3-h.
\]
Exhausting the representation resource means $\lambda_2+\lambda_3=h$, hence
$c\log12=3h$.  The proof uses only strict inequalities and takes the
endpoint only after the density-one statement has been established
for every smaller $c$.

\begin{remark}[Relation to prior work]\label{rem:comparison}
Pomerance's Theorem~1 \cite{Pomerance2026} and Sothanaphan's Appendix
Corollary~2 \cite{Sothanaphan2026} obtain a centered rising-product
divisibility range with coefficient $1/\log4$.  Their carry-rich,
spike-controlled mechanism works inside one middle binomial
coefficient.  Earlier central-binomial work of
Erd\H{o}s--Graham--Ruzsa--Straus, Pomerance, Sanna, and
Ford--Konyagin concerns prime factors or prescribed divisors of
$\binom{2m}{m}$ and supplies important background
\cite{ErdosGrahamRuzsaStraus1975,Pomerance2015,Sanna2018,FordKonyagin2021}.
The mixed-resource representation developed here is a distinct
mechanism.

The digital input is likewise distinct from existing simultaneous-base
results.  Spiegelhofer and Spiegelhofer--Stoll study binary--ternary
collisions and digit sums on arithmetic progressions
\cite{Spiegelhofer2023,SpiegelhoferStoll2020}.  Croot--Mousavi--Schmidt
and Bloom--Croot construct integers satisfying simultaneous digit or
carry restrictions in several bases
\cite{CrootMousaviSchmidt2024,BloomCroot2025}.  The present argument
instead requires normal-order digit mass along affine progressions
whose multiplier is an $S$-unit that grows as a fixed power of the
parameter scale.

Drmota--Spiegelhofer's Lemma~3.4 fixes the coefficients before the
threshold is introduced.  Our Fourier argument requires one threshold
for polylogarithmically varying coefficients, so \cref{lem:phase}
derives the required coefficient-uniform statement for $d\le2$
directly from their Lemma~3.3 by resolving its finite exceptional
family.
\end{remark}

\begin{remark}[Quantitative status]\label{rem:effectivity}
The elementary concentration, two-block, representation, and
large-prime estimates are quantitative.  The density threshold inherits
its qualitative character from the finite exceptional-subspace
conclusion used in \cref{thm:sunitnormal}.  Effective sparse-form
estimates would yield an explicit threshold in \cref{thm:main}.
\end{remark}

\begin{question}\label{q:constant}
Is $(k-1)/\log2$ the true density-one leading constant?  Reaching it
would require nearly maximal binary carry mass while simultaneously
meeting ternary and higher-prime valuation constraints.
\end{question}

\begin{question}\label{q:clt}
Can \cref{thm:sunitnormal} be upgraded, for arbitrary fixed disjoint
prime sets, to a uniform central or local limit theorem?  Drmota and
Spiegelhofer prove much stronger binary--ternary results, while the
present proof is built from one- and two-digit correlations.
\end{question}

\begin{question}\label{q:exceptional}
Can one obtain an effective power-saving estimate for the exceptional
set in \cref{thm:main} by replacing the qualitative
Subspace-Theorem step with an effective sparse-form estimate?
\end{question}

\appendix
\section{Parameter ledger}\label{app:ledger}

For reference, the choices in the proof of \cref{thm:main} occur in
the following order.
\begin{enumerate}
\item Fix $k\ge2$ and $c<3(k-1)/\log12$.
\item Choose $c<c_1<c_0<3(k-1)/\log12$ and resources $\lambda_2,\lambda_3$ with
strict binary, ternary, and representation margins.
\item Split the resources into finitely many exponents
$e_i\in(0,1)\setminus\{1/2\}$ and define the fixed separation margin
$\delta_{\rm exp}$.
\item Choose the fixed-prime digit error $\rho$ below one quarter of
all binary, ternary, and quinary margins.
\item Choose the medium-prime error
$\eta<\Delta/[16(k+1)]$, where
$\Delta=k-1-\lambda_2-\lambda_3$.
\item Choose the fixed cutoff $P$ so that
$c_1\log p/(p-1)<\Delta/3$ for $p>P$.
\item Choose $B\ge4$ and use the two-block estimate with logarithmic
error exponent $D=B+4$.
\item Finally take $X$ sufficiently large for all uniform estimates,
all fixed target subintervals, and all rounding errors.
\end{enumerate}
No limit in $P$, $B$, or an error parameter is interchanged with the
limit $X\to\infty$.

\section{Application-by-application hypothesis audit}
\label{app:hypothesisaudit}

For transparency, we list the parameters used at every nontrivial
lemma invocation in the proof of \cref{thm:main}.  This appendix is
part of the proof bookkeeping; no computation is invoked.

\begin{enumerate}
\item \emph{Representation lemma.}
The representation lemma is applied with its coordinate count equal to $k$.  The target interval is
$[X/2,3X]$.  Each multiplier satisfies
$M_i=X^{e_i+O(1/\log X)}$ with fixed $e_i\in(0,1)$, and
$U_i=X/M_i$.  Here $r$ is the number of binary multipliers, and the distinguished pair is
$(M_1,M_{r+1})=(2^{J_1},3^{K_1})$; it is coprime and its product is
$o(X)$ because $x_1+y_1<1$.  Thus all hypotheses of
\cref{lem:representations} hold.  Equation
\eqref{eq:representationbox} is used on the full parameter box;
\eqref{eq:representationgeometry} is used only after imposing the
representation equation.

\item \emph{Fixed-prime $S$-unit normal order.}
For a coordinate with multiplier $M_i$, take $A=M_i$, $b=-1$,
$U=U_i$, and $I=I_i$.  The finite set $S$ is $\{2\}$, $\{3\}$, or
$\{2,3\}$ according to the multiplier family, and the target-prime
set is fixed before $X\to\infty$.  Equations
\eqref{eq:Jbound}--\eqref{eq:Kbound} give one constant $C_*$ with
$A\le U_i^{C_*}$.  The representation construction places $I_i$ in a
fixed proportional interval at scale $U_i$, while
\eqref{eq:fullboxparts} gives $M_it-1>0$ for every $t\in I_i$.
Hence every hypothesis of \cref{thm:sunitnormal} holds uniformly.

\item \emph{Forced suffix.}
For a same-base multiplier, take $e=e_i$ and $p^J=M_i$.  By
\eqref{eq:exponentgap}, $1-e_i\ge4\delta_{\rm exp}$; hence the growth
hypothesis in \cref{lem:forcedsuffix} holds with
$\delta=2\delta_{\rm exp}$.  Moreover $p^JU_i=X$ exactly and $I_i$ is
a proportional interval of length $\asymp U_i$.

\item \emph{Two-block affine estimate.}
For coordinate $i$, take $a=e_i$, $u=1-e_i$, $M=M_i$, $U=U_i$,
$I=I_i$, and $d=-1$.  Equation \eqref{eq:exponentgap} supplies a
fixed positive lower bound for $a,u,$ and $|a-u|$.  The multiplier and
parameter scales have only bounded, hence admissible, rounding
factors.  The entire-box comparison
\eqref{eq:fullboxparts} gives $M_it-1\asymp X$ for every parameter in
$I_i$.  For medium primes, $p>P\ge5$ and therefore $p\nmid M_i$.
Thus \cref{lem:twoblock} applies simultaneously on the full box, not
merely to represented tuples.

\item \emph{Target digit upper bounds.}
Every target lies in $[X,2X)$, and the prime range is at most
$(\log X)^B$.  Thus \cref{lem:targetupper} is applied with fixed
proportional constants and $A=B$.  Its exceptional target set is
$o(X)$ independently of the representation parameters.

\item \emph{Large-prime sieve.}
All $M_i$ have prime factors in the fixed set $\{2,3\}$.  From the
fixed exponents, $U_i\gg X^{\delta_L}$ for
$\delta_L=\min_i(1-e_i)/2>0$.  Equation
\eqref{eq:fullboxparts} gives $a_i=M_it_i-1\asymp X$ throughout the
full box, and $v=O(\log X)$.  Hence \cref{lem:largeprime} applies with
the already fixed $B\ge4$.  On represented tuples its internally
defined integer $n=\sum_i a_i-v$ is the prescribed factorial target.

\item \emph{Averaging over representations.}
Let $V=\prod_i|I_i|$.  Each target has at least $c_RV/X$
representations and at most $C_RV/X$.  Coordinate exceptions,
two-block exceptions, and large-prime exceptions together have size
$o(V)$.  The $o(X)$ abnormal targets contribute at most
$o(X)\,O(V/X)=o(V)$ represented tuples.  Therefore
\cref{cor:representationaveraging} leaves a good representation for
all but $o(X)$ targets in each fixed subinterval.
\end{enumerate}

\section*{Acknowledgements}
The author used OpenAI's ChatGPT during the preparation of this manuscript for ideation, formulation, proof exploration and refinement, narrowing the search space, programming, LaTeX formatting, and related orchestration. The author is responsible for the final contents of the paper.

I thank SamKorsky for the independent discovery and public announcement (20 June 2026) of the matching density-one lower bound with the same leading coefficient \(\frac{3(k-1)}{\log 12}\). This confirms that the value is the natural limit of the mixed binary--ternary resource method under the representation constraint.

\end{document}